\documentclass[letterpaper,11pt]{article}
\usepackage[latin2, utf8x]{inputenc}
\usepackage{amsfonts}
\usepackage{amsthm}
\usepackage{amsmath}
\usepackage{color}
\usepackage{tikz}
\usepackage{graphicx}

\newcommand{\RP}[1]{\mathbb{P}^{#1}_{\mathbb{R}}}

\newcommand{\mat}[2]{
  \left(
  \begin{array}{#1}
    #2
  \end{array}
  \right)
}

\newcommand{\Mat}[2]{
  \left[
  \begin{array}{#1}
    #2
  \end{array}
  \right]
}

\newcommand{\change}{
\begin{center}
 $\ast \ast \ast$
\end{center}
}

\newtheorem{prop}{Fun Fact}
\newtheorem{coro}{Corollary}

\newtheorem{question}{Question}
\newtheorem{conjecture}{Conjecture}

\title{A Multidimensional Gauss Map}
\author{Jes\'us Hern\'andez Serda \\ \texttt{jesus.hernandez@im.unam.mx}}
\date{October 31, 2016}

\begin{document}

\maketitle

\section{Introduction}

The continued fraction expansion gives a one-to-one correspondence between irrational numbers and sequences of natural numbers. 
This correspondence shows some properties of irrational numbers in terms of the sequences that represents them, for example, bounded sequences represent irrational numbers that satisfy a \emph{Diophantine} condition, and eventually periodic sequences represent exactly the quadratic algebraic irrational numbers. 
An age old question, the Hermite Problem, asks whether or not there exists a characterization of the sequences that represent algebraic numbers of higher degree or if there is another correspondence between irrational numbers and sequences of symbols for which periodic sequences represent algebraic numbers of a given degree. 

A dynamical way of approaching these questions is to study the dynamical systems with the real numbers as the base space which admit a conjugation to a symbolic system. Each system gives us a correspondence between a set of numbers and the sequences of symbols of the associated shift space. For example, the binary expansion of real numbers can be retrieved from the map $x \mapsto 2x \mod 1$. 
The classical Gauss Map is a piecewise continuous map from the unit interval to itself. From this map we retrieve the continued fraction expansion of irrational numbers and its dynamical properties give information about some arithmetic and algebraic properties of irrational numbers. In this notes we will explore some generalizations of the Gauss Map to higher dimensions and pose some questions and conjectures about the arithmetic/algebraic information that these maps may carry.

For the construction of this generalizations we will be using the projective spaces $\RP{n} = \mathbb{R}^{n+1} / \mathbb{R}^{\ast}$ and the homogeneous coordinate notation. 
A point $\mathbf{v} \in \RP{n}$ corresponds to a straight line going through the origin in $\mathbb{R}^{n+1}$ and it will be denoted as \[ \mathbf{v} = \Mat{c}{x_1 \\ \vdots \\ x_{n+1}}. \]
We will also be using \textemdash{}mostly in an implicit way\textemdash{} the canonical chart $\mathbb{R}^{n} \hookrightarrow \RP{n}$, given by \[ (x_1, ..., x_{n}) \mapsto \Mat{c}{x_1 \\ \vdots \\ x_{n} \\ 1}\] and the projection 
\[ \Mat{c}{x_1 \\ \vdots \\ x_{n+1}} \mapsto \left( \displaystyle\frac{x_1}{x_{n+1}}, ..., \displaystyle\frac{x_{n}}{x_{n+1}}\right) \hspace*{2em}\text{ if }\hspace*{2em} x_{n+1} \neq 0 \]

For the maps in this notes, a simplex will be taking the place of the unit interval. We will denote a $n$-dimensional simplex in $\RP{n}$ as a $(n+1) \times (n+1)$ matrix in brackets, with the column vectors of the matrix representing the vertices of the simplex in homogeneous coordinates.   

The maps we will consider are \emph{piecewise linear} in the sense that there is a countable partition of the base simplex in sub-simplexes and a matching family of $(n+1)\times(n+1)$ matrices. These maps act on a point by multiplying the matrix that corresponds to the sub-simplex containing the point. Note that these maps are linear in $\mathbb{R}^{n+1}$ but not necessarily when projected to $\RP{n}$. 

Since some of the properties we are studying are shared by whole orbits, we can consider the group generated by the defining matrices of one of these \emph{piecewise linear} maps and take orbits of the action of this group. Note that two points belong to the same orbit for the \emph{piecewise linear} map if there is an element of the associated group taking one into the other.

\section{Continued Fractions and the Gauss Map}

Consider the following matrices and simplex 
\[ A = \mat{cc}{1 & 0 \\ -1 & 1}, \hspace*{2em}
   B = \mat{cc}{-1 & 1 \\ 1 & 0}, \hspace*{2em}
   V = \Mat{cc}{0 & 1 \\ 1 & 1}.\]
Note that $A, B \in SL(2, \mathbb{Z})$ and the simplex represented by $V$ is the unit interval. The classic Farey Map $F: V \rightarrow V$ 
is given in homogeneous coordinates as 
\[ F (\mathbf{v}) = \left\lbrace \begin{array}{lcc} 
   A \mathbf{v} & \text{ if } & \mathbf{v} \in A^{-1}V \\ 
   B \mathbf{v} & \text{ if } & \mathbf{v} \in B^{-1}V \end{array} \right. ,\]
and when projected to $\mathbb{R}$ it takes the form 
\[ F(x) = \left\lbrace \begin{array}{lcc} 
        x/(1 - x) & \text{ if } & 0 \leq x \leq \frac{1}{2} \\ 
        (1 - x)/x & \text{ if } & \frac{1}{2} \leq x \leq 1 \end{array} \right. .\] 
Note that $V = A^{-1}V \cup B^{-1}V$ and this sub-simplexes are the intervals $[0, 1/2]$ and $[1/2, 1]$ respectively.

Consider the first return map to the simplex $B^{-1}V$: 
For a point $\mathbf{v} \in V$ let $k_{\mathbf{v}}$ be the least non-negative integer such that $F^{k_{\mathbf{v}}} (\mathbf{v}) \in B^{-1}V$. For example, if $\mathbf{v} \in B^{-1}V$ then $k_{\mathbf{v}} = 0$, on the other hand, if $\mathbf{v} \notin B^{-1}V$ then $k_B(\mathbf{v}) > 0$ and $F$ acts on $\mathbf{v}$ as multiplication by $A$. 
The first return map is defined as
\[ G (\mathbf{v}) = \left\lbrace \begin{array}{lcc} 
   B A^{k_{\mathbf{v}}} \mathbf{v} & \text{ if } & \mathbf{v} \neq \Mat{c}{0 \\ 1} \\ 
   \Mat{c}{0 \\ 1} & \text{ if } & \mathbf{v} = \Mat{c}{0 \\ 1} \end{array} \right. .\] 
This map $G$ is defined piecewise on $(0,1]$ because the only fixed point of $A$ in $V$ is zero, and $A$ is expanding on $(0,1/2]$. 
We can retrieve the domains of definition of $G$ in a similar way as for the Farey Map. 
For $n \geq 1$ let $\mathbf{A}_n = B A^{n-1}$, then the set of points that take $n-1$ iterations of $F$ to get to the simplex $B^{-1}V$ are in the simplex $\mathbb{A}_n = \mathbf{A}_n^{-1} V$. 
Note that for all $n \in \mathbb{Z}$ 
\[ A^{n} = \mat{cc}{1 & 0 \\ -n & 1},\]
and that gives an expression for the simplexes $\mathbb{A}_n $  
\[ \mathbb{A}_n = A^{1-n} B^{-1} V = \mat{cc}{1 & 0 \\ n-1 & 1} \Mat{cc}{1 & 1 \\ 1 & 2} = \Mat{cc}{1 & 1 \\ n & n+1}, \]
and for the matrices $\mathbf{A}_n$
\[ \mathbf{A}_n = \mat{cc}{-1 & 1 \\ 1 & 0} \mat{cc}{1 & 0 \\ 1-n & 1} = \mat{cc}{-n & 1 \\ 1 & 0} .\]
The map $G$ can be written as
\[ G (\mathbf{v}) = \left\lbrace \begin{array}{lcc} 
   \mathbf{A}_n \mathbf{v} & \text{ if } & \mathbf{v} \in \mathbb{A}_n \\ 
   \Mat{c}{0 \\ 1}& \text{ if } & \mathbf{v} = \Mat{c}{0 \\ 1} \end{array} \right. .\] 
and when projected to $\mathbb{R}$, the map $G$ takes the form of the classic Gauss Map
\[ G (x) = \left\lbrace \begin{array}{lcc} 
   \frac{1}{x} - n & \text{ if } & \frac{1}{n+1} < x \leq \frac{1}{n} \\ 
   0 & \text{ if } & x = 0 \end{array} \right. 
   \hspace*{2em} \text{or} \hspace*{2em} 
   G (x) = \left\lbrace \begin{array}{lcc} 
   \frac{1}{x} \mod 1 & \text{ if } & x \neq 0 \\ 
   0 & \text{ if } & x = 0 \end{array} \right. .\] 

In the Figure \ref{graph_farey_gauss} we have the graphs of the Farey and Gauss map when projected to the unit interval.
Note that some of the domains share an end point and the corresponding transformations do not match in this common point, so it is necessary to take the usual convention $\mathbb{A}_n = (1/(n+1), 1/n]$ in order to get the preferred expression in continued fractions for the rational numbers (the one that doesn't end with the symbol $\mathbb{A}_1$).
Note that with this convention there are no preimages of $1$.

Recall that irrational numbers have unique expressions as regular continued fractions and rational numbers have two possible expressions and both are finite:
\[ [a_0 : a_1 , a_2, ...] = a_0 + \cfrac{1}{a_1 + \cfrac{1}{a_2 + \cfrac{1}{\ddots}}} \hspace*{2em} \text{and} \hspace*{2em} [a_0: a_1, ..., a_n] = [a_0: a_1, ..., a_n - 1, 1] .\]
The action of the Gauss Map on the continued fraction expansion is the shift, i.e. $G([0: a_1, a_2, ...]) = [0: a_2, ...]$

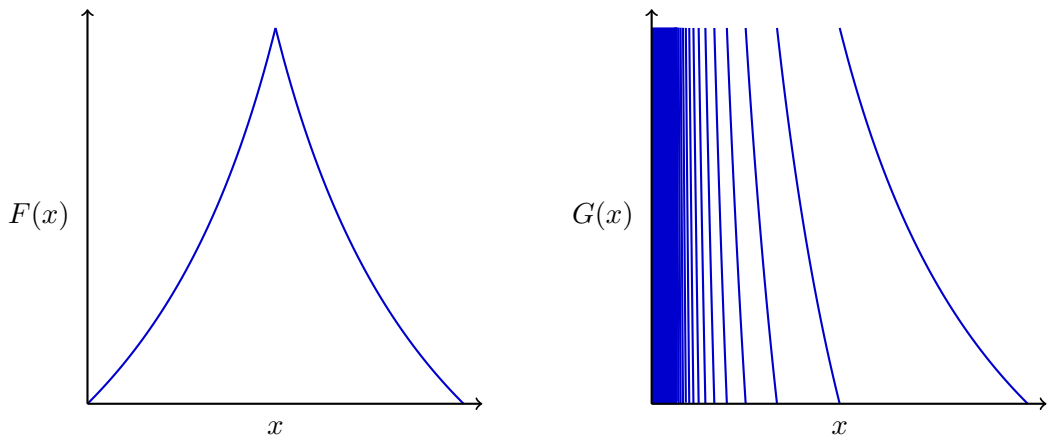
\begin{figure}
\begin{center}
 \begin{tikzpicture}[scale=0.5]
 
 \draw[blue!80!black, thick, domain=-15:-10, smooth, variable=\x] plot ({\x},{(10*\x + 150) / (-5 - \x)});
 \draw[blue!80!black, thick, domain=-10: -5, smooth, variable=\x] plot ({\x},{(-50 - 10*\x) / (\x + 15)});
 
 \draw[->,thick] (-15,0) --(-15,10.5);
 \draw[->,thick] (-15,0) --(-4.5,0);
 
 \node[black, left] at (-15.2, 5) {$F(x)$};
 \node[black, below] at (-10, -0.2) {$x$};

  \foreach \y in {1,...,15}
     {
        \draw[blue!80!black, thick, domain=10/(\y+1):10/(\y), smooth, variable=\x] plot ({\x},{(100/\x) - 10*\y});
     }
  \draw [fill=blue!80!black,blue!80!black] (0,0) rectangle (0.66667,10); 
  
  \draw[->,thick] (0,0) --(10.5,0);
  \draw[->,thick] (0,0) --(0,10.5);
  
  \node[black, below] at (5, -0.2) {$x$};
  \node[black, left] at (-0.2, 5) {$G(x)$};

 \end{tikzpicture}
 \caption{Graphs of the Farey Map and the Gauss Map}
 \label{graph_farey_gauss}
\end{center}
\end{figure}

Now we review some properties of the Gauss Map and the continued fractions using our notation. 
Let's begin with the associated group. 

\begin{prop}
 The group generated by the matrices $A$ and $B$ is $SL(2,\mathbb{Z})$.
\end{prop}

\begin{proof}
 It is known that $SL(2, \mathbb{Z})$ is generated by the matrices \[S = \mat{cc}{0 & -1 \\ 1 & 0} \hspace*{2em} T = \mat{cc}{1 & 1 \\ 0 & 1} \] and since we can write these matrices as $T = B A^{-1} B^{-1}$ and $S = (T A T)^{-1}$ we have that \[\langle A, B \rangle = \langle S, T \rangle = SL(2, \mathbb{Z}).\] 
\end{proof}

By a similar argument we have that $\langle \lbrace \mathbf{A}_n : n \geq 1 \rbrace \rangle = SL(2, \mathbb{Z})$. 
Even though the action of $G$ (and $F$) is defined only in the simplex $V$, the action of $SL(2, \mathbb{Z})$ is defined on all $\RP{1}$. 
In this case, the simplex $V$ represents all real numbers up to integral part and the action of $SL(2, \mathbb{Z})$ relates every real number with its fractional part:
Let $\alpha \in \mathbb{R}$ and $n_0 \in \mathbb{Z}$ such that $n_0 \leq \alpha < n_0 + 1$. 
We have that \[\mat{cc}{1 & -n_0 \\ 0 & 1} \in SL(2, \mathbb{Z}) \hspace*{2em} \text{and} \hspace*{2em} \mat{cc}{1 & -n_0 \\ 0 & 1} \Mat{c}{\alpha \\ 1} = \Mat{c}{\alpha - n_0 \\ 1} \in V .\]

\begin{prop}
 For the map $G$ all rational points are preimages of zero.
\end{prop}

\begin{proof}
 A rational number $p/q \in [0,1]$ is represented in homogeneous coordinates by a vector with integer entries \[\mathbf{v} = \Mat{c}{p \\ q}.\] 
 Suppose $GCD (p, q) = 1$, then there exists $a,b \in \mathbb{Z}$ such that $a p + b q = 1$. Consider the matrix 
 \[\mathbf{M} = \mat{cc}{b & p \\ -a & q}.\]
 We have that $\mathbf{M} \in SL(2, \mathbb{Z})$ by the previous condition and also 
 \[ \mathbf{M} \Mat{c}{0 \\ 1} = \mathbf{v} \]
 so $\mathbf{v}$ and $0$ are in the same orbit for $G$.
\end{proof}

Consider the simplex $V$ embedded in $\mathbb{R}^{2}$ as the convex hull of the points $(0,1)$ and $(1,1)$. For a number $\alpha \in [0,1]$ let $\mathbf{v}$ be the straight line in $\mathbb{R}^2$ going through the origin and the point $(\alpha, 1)$, this is the corresponding point in $\RP{1}$. 
When $\alpha \in \mathbb{Q}$ the line $\mathbf{v}$ passes through infinitely many points in $\mathbb{Z}^2$, with the closest to the origin having as coordinates the numerator and denominator of $\alpha$ in lowest terms. 
When $\alpha$ is irrational the line $\mathbf{v}$ does not intersect $\mathbb{Z}^2$. 

For any irrational number $\alpha$ there is an approximation by rational numbers given by the continued fractions, the convergents: for $\alpha = [0: a_1, a_2, ...]$ the $n$-th convergent is the rational number $\frac{p_n}{q_n} = [0: a_1, ..., a_n]$.
This approximation is the \emph{best} in the sense that if any rational number $a/b$ is closer to $\alpha$ than a convergent $p_n/q_n$, then $b > q_n$.
We will state this property in terms of \emph{approximating simplexes}. 
For each $n \geq 1$ let $\mathbf{I}_{\mathbf{v}, n} = \left( \mathbf{A}_{a_1}^{-1} \mathbf{A}_{a_2}^{-1} \cdots \mathbf{A}_{a_n}^{-1} \right)$. 
The point $\mathbf{v}$ is in the simplex $\mathbb{I}_{n}(\mathbf{v}) = \mathbf{I}_{\alpha,n} V$.
We have by construction that $\mathbb{I}_{n+1}(\mathbf{v}) \subset \mathbb{I}_{n}(\mathbf{v})$ and by the uniqueness of the continued fractions $\mathbf{v} = \bigcap_{n \geq 1} \mathbb{I}_{n}(\mathbf{v})$.
We call the sequence $\mathbb{I}_{n}(\mathbf{v})$ the approximating simplexes of $\mathbf{v}$.

\begin{prop}
 The matrices $\mathbf{I}_{\mathbf{v}, n}$ satisfy \[ \mathbf{I}_{\mathbf{v}, n} = \mat{cc}{p_{n-1} & p_{n} \\ q_{n-1} & q_{n}} .\]
\end{prop}

\begin{proof}
 For any irrational number $\alpha = [0: a_1, a_2, ...]$ the $0$-th convergent is zero and the $1$-st convergent is $1/a_1$, so for the case $n=1$ we have \[ \mathbf{I}_{\mathbf{v},1} = \mathbf{A}_{a_1}^{-1} = \mat{cc}{0 & 1 \\ 1 & a_1} = \mat{cc}{p_{0} & p_{1} \\ q_{0} & q_{1}}. \]
 Now suppose the claim is valid for $n-1$,  we can write $\mathbf{I}_{\mathbf{v}, n} = \mathbf{I}_{\alpha,n-1} \mathbf{A}_{a_n}^{-1}$ and we have that \[ \mathbf{I}_{\mathbf{v}, n} = \mat{cc}{p_{n-2} & p_{n-1} \\ q_{n-2} & q_{n-1}} \mat{cc}{0 & 1 \\ 1 & a_n} = \mat{cc}{p_{n-1} & a_n p_{n-1} + p_{n-2}\\ q_{n-1} & a_n q_{n-1} + q_{n-2}} = \mat{cc}{p_{n-1} & p_{n} \\ q_{n-1} & q_{n}}, \]
 with the last equality given by the recurrence of the convergents, see \cite{Khinchin}.
\end{proof}

\begin{coro}
 The approximating simplexes are of the form \[ \mathbb{I}_{n}(\mathbf{v}) = \Mat{cc}{ p_n & p_n + p_{n-1} \\ q_n & q_n + q_{n-1} }, \]
 i.e. the columns of $\mathbb{I}_{n}(\mathbf{v})$ are the $n$-th convergent and the \emph{mediant}\footnote{The \emph{mediant} of two rational numbers $a/b$, $c/d$ is $(a+c)/(b+d)$. } of the $n$-th and the $(n-1)$-th convergents in homogeneous coordinates.
\end{coro}

\begin{coro}[Twisted nesting]
 For consecutive convergents we have either \[ \frac{p_n}{q_n} < \frac{p_n + p_{n+1}}{q_n + q_{n+1}} < \alpha < \frac{p_{n+1}}{q_{n+1}} < \frac{p_n + p_{n-1}}{q_n + q_{n-1}} \hspace*{2em}\text{ or }\hspace*{2em} \frac{p_n}{q_n} > \frac{p_n + p_{n+1}}{q_n + q_{n+1}} > \alpha > \frac{p_{n+1}}{q_{n+1}} > \frac{p_n + p_{n-1}}{q_n + q_{n-1}}. \]
\end{coro}

We get the last corollary from the facts $\mathbb{I}_{n+1}(\mathbf{v}) \subset \mathbb{I}_{n}(\mathbf{v})$ and $\det \mathbf{I}_{\mathbf{v}, n} = (-1)^{n}$.

For each of the approximating simplexes $\mathbb{I}_{n}(\mathbf{v})$ consider the triangle in $\mathbb{R}^2$ with vertices in the origin and the two points in $\mathbb{Z}^2$ which give the end points of $\mathbb{I}_{n}(\mathbf{v})$, i.e. $(p_n, q_n)$ and $(p_n + p_{n-1}, q_n + q_{n-1})$. Suppose there is a point of $\mathbb{Z}^2$ inside this triangle, that would mean there is a straight line passing through such point, and that line would correspond to a rational number closer to $\alpha$ than the convergent $p_n/q_n$.
Moreover, being inside the triangle would imply that the denominator of such rational number is smaller than the denominator of the convergents, which would imply this number approximates $\alpha$ \emph{better} than the convergents.

\begin{prop}[Best approximations]
 For all irrational numbers and any approximating simplex the related triangle does not contain points of $\mathbb{Z}^2$ other than its vertices.
\end{prop}

\begin{proof}
 This is a direct result of \emph{Pick's Theorem} which states that the area of a polygon with vertices in $\mathbb{Z}^2$ is exactly \[ \text{Area} = \iota + \frac{\kappa}{2} - 1,\] where $\iota$ is the number of interior integral points and $\kappa$ is the number of boundary integral points. 
 If the polygon is a triangle with area $1/2$ there are at least three boundary points (the vertices) and  \emph{Pick's Theorem} gives that \[ \frac{1}{2} = \iota  + \frac{\kappa' + 3}{2} - 1 \hspace*{1em} \Longrightarrow \hspace{1em} \frac{\kappa'}{2} + \iota = 0.\]
 In other words, the number of interior integral points and non-vertex boundary integral points is zero.
 
 The approximating simplexes are all of the form $\mathbf{M} V$ with $\mathbf{M} \in SL(2, \mathbb{Z})$. 
 The triangle with vertices in the origin and the end points of $V$ has area $1/2$ and its only integral points are boundary points: $(0,0)$, $(0,1)$ and $(1,1)$. 
 Since $SL(2,\mathbb{Z})$ preserves the area of these triangles, none of this triangles contain integral points other than its vertices.
\end{proof}

\change

\begin{prop}[Rate of convergence]
 For any irrational number $\alpha$ the sequence of convergents satisfies \[ \frac{1}{2 q_{n+1}^2} < \frac{1}{q_n (q_n + q_{n+1})} < \left| \alpha - \frac{p_n}{q_n} \right| < \frac{1}{q_n q_{n+1}} < \frac{1}{q_n^2}\]
\end{prop}

\begin{proof}
 We get the first and last inequalities from the sequence $q_n$ being increasing. 
 Let $\mathbf{v} = \Mat{c}{\alpha \\ 1}$.
 Recall that for all $n>0$ the number $\alpha$ lies between the $n$-th and $(n+1)$-th convergents, i.e. \[ \left| \alpha - \frac{p_n}{q_n} \right| < \left| \frac{p_{n+1}}{q_{n+1}} - \frac{p_n}{q_n} \right|.\]
 Now, $|\det \mathbb{I}_{n+1}(\mathbf{v})| = |p_{n+1} q_n -  p_n q_{n+1}| = 1$, and this implies that \[ \left| \frac{p_{n+1}}{q_{n+1}} - \frac{p_n}{q_n} \right| = \frac{1}{q_n q_{n+1}}.\]
 Finally, from the \emph{twisted} nesting $\mathbb{I}_{n+1}(\mathbf{v}) \subset \mathbb{I}_{n}(\mathbf{v})$ we have that \[ \left| \alpha - \frac{p_n}{q_n} \right| > \left| \frac{p_{n+1} + p_{n}}{q_{n+1}+ q_{n}} - \frac{p_n}{q_n} \right| =  \frac{\left|\det \mathbb{I}_{n+1}(\mathbf{v})\right|}{q_n (q_n + q_{n+1})} .\]
\end{proof}


 Now consider the periodic points of $G$. 
 Let $\mathbf{v} \in V$ be a non-rational point and $k > 0$ be an integer such that $\mathbf{v} = G^{k} (\mathbf{v})$. 
 There is a matrix $M \in SL(2, \mathbb{Z})$ satisfying $M \mathbf{v} = \mathbf{v}$, in other words, $\mathbf{v}$ is an eigenvector of $M$. 
 This equation can be written as
 \[\mat{cc}{a & b \\ c & d} \Mat{c}{x \\ 1} = \Mat{c}{ax+b \\ cx+d} = \Mat{c}{x \\ 1},\]
 and this implies that $x$ is a root of a quadratic polynomial with integer coefficients:
 \[\frac{ax+b}{cx+d} = x \hspace*{2em} \Longrightarrow \hspace*{2em} cx^2 - ax + dx -b = 0.\]

 \begin{prop}
  Let $0 < x < 1$ be a quadratic algebraic number. 
  The point $\Mat{c}{x\\1}$ is eventually periodic.
 \end{prop}

 \begin{proof}
  Let $x = [0: a_1, a_2, ...]$ be an irrational such that $ax^2 + bx + c = 0$ for some $a,b,c \in \mathbb{Z}$, and \[\mathbf{v} = \Mat{c}{x\\1}. \hspace*{2em} \text{Note that } \mat{cc}{x & 1} \mat{cc}{a & b/2 \\ b/2 & c} \mat{c}{x\\1} = 0.\]
  We will show that the set $\lbrace G^{k}(\mathbf{v}) : k \geq 0 \rbrace$ is finite and therefore $\mathbf{v}$ is eventually periodic.
  
  For each $k$, let $x_k = [0: a_k, a_{k+1}, ...]$, then \[G^{k} (\mathbf{v}) = \Mat{c}{x_k\\1} = \mathbf{A}_{a_k} \cdots \mathbf{A}_{a_1} \Mat{c}{x\\1},\]
  \[\mathbf{A}_{a_1}^{-1} \cdots \mathbf{A}_{a_k}^{-1} \Mat{c}{x_k\\1} = \mat{cc}{p_{k-1} & p_k \\ q_{k-1} & q_k}  \Mat{c}{x_k\\1} = \Mat{c}{x\\1} 
  .\]
  Substituting in the quadratic equation we get that $x_k$ must satisfy the following:
  \[
    \mat{cc}{x_k & 1} \mat{cc}{p_{k-1} & q_{k-1} \\ p_k & q_k} \mat{cc}{a & b/2 \\ b/2 & c} \mat{cc}{p_{k-1} & p_k \\ q_{k-1} & q_k} \mat{c}{x_k\\1} = 0,
  \]
  {\tiny
  \[
    \mat{cc}{x_k & 1} \mat{cc}{a p_{k-1}^2 + b p_{k-1} q_{k-1} + c q_{k-1}^2 & a p_k p_{k-1} + \frac{b}{2} (p_k q_{k-1} + q_k p_{k-1}) + c q_k q_{k-1} \\ a p_k p_{k-1} + \frac{b}{2} (p_k q_{k-1} + q_k p_{k-1}) + c q_k q_{k-1} & a p_k^2 + b p_k q_k + c q_k^2} \mat{c}{x_k\\1} = 0.
  \]
  }
  As we will see, the coefficients of the resulting quadratic equation take finitely many different values, which implies that the numbers $x_k$ are roots of finitely many quadratic polynomials and therefore $x_k$ takes finitely many different values.
  
  For each convergent we have that \[\left| \frac{p_k}{q_k} - x \right| < \frac{1}{q_k^2}.\] 
  So we can find numbers $|\varepsilon_k| < 1$ such that \[x = \frac{p_k}{q_k} + \frac{\varepsilon_k}{q_k^2}.\]
  Denote the coefficients of the resulting quadratic equation as 
  \[
   \begin{split}
    \tilde{a} &= a p_{k-1}^2 + b p_{k-1} q_{k-1} + c q_{k-1}^2, \\
    \tilde{b}/2 &= a p_k p_{k-1} + \frac{b}{2} ( p_k q_{k-1} + q_k p_{k-1}) + c q_k q_{k-1}, \\
    \tilde{c} &= a p_k^2 + b p_k q_k + c q_k^2.
   \end{split}
  \]
 For $\tilde{c}$ we have that 
 \[
 \begin{split}
  \frac{\tilde{c}}{q_k^2} &= a \left(\frac{p_k}{q_k}\right)^2 + b \left(\frac{p_k}{q_k}\right) + c = a \left(x - \frac{\varepsilon_k}{q_k^2}\right)^2 + b \left(x - \frac{\varepsilon_k}{q_k^2}\right) + c ,\\
  \tilde{c} &= a \frac{\varepsilon_k^2}{q_{k}^2} - (2 a x + b) \varepsilon_k ,\\
  |\tilde{c}| &\leq |a| + |2ax + b|.
 \end{split}
 \]
 And $\tilde{c}$ can only take finitely many different values; for $\tilde{a}$ the proof is the same.
 Note that 
 \[
   \det \mat{cc}{a&b/2\\b/2&c} = \det \mat{cc}{\tilde{a} & \tilde{b}/2 \\ \tilde{b}/2 & \tilde{c}},
 \]
 since the matrices of the convergents are in $SL(2,\mathbb{Z})$. Therefore $\tilde{b}$ also takes finitely many different values.
 \end{proof}

 \change
 


Some dynamical properties of the Gauss Map reflect arithmetic properties of irrational numbers.  
The arithmetic property we recall is the notion of \emph{Irrationality Measure}. For a given irrational number $\alpha$, let $\mathcal{D} \subset \mathbb{R}^+$ be the set of numbers $\gamma$ such that there are only finitely many rational numbers $p/q$ satisfying the inequality \[ \left| \alpha - \frac{p}{q} \right| \leq \frac{1}{q^{2+\gamma}}. \]
Note that if $\gamma \in \mathcal{D}$ then $[\gamma, \infty) \subseteq \mathcal{D}$. The irrationality measure of $\alpha$ is defined as \[ \mu(\alpha) = \inf_{\gamma \in \mathcal{D}} \gamma,\] setting $\mu(\alpha) = \infty$ when $\mathcal{D}$ is empty.

The dynamical quantity related to the Irrationality Measure is the \emph{upper} Lyapunov Exponent for the Gauss Map.  
For an irrational number $\alpha \in [0,1]$ the \emph{upper} Lyapunov exponent is defined as \[ \lambda^{+}(\alpha) = \limsup_{n \rightarrow \infty} \frac{1}{n} \log \left|(G^n)' (\alpha)\right| = \limsup_{n \rightarrow \infty} \frac{1}{n} \sum_{j=0}^{n-1} \log \left|G'(G^j (\alpha))\right|, \] with the equivalent expressions \[ \lambda^{+} (\alpha) = - \limsup_{n \rightarrow \infty} \frac{1}{n} \log \left| \alpha - \frac{p_n}{q_n} \right| = \limsup_{n \rightarrow \infty} \frac{2}{n} \log q_n.\]

\begin{prop}
 If $\lambda^{+}(\alpha) < \infty$ then $\mu(\alpha) = 0$.  
\end{prop}

The relation between arithmetics and dynamics comes from the fact that the number $\lambda^{+}(\alpha)/2$ is an upper bound for the exponential growth rate of the sequence $q_n$. 
It is known that irrational numbers of bounded type, i.e. numbers whose continued fraction expansion has bounded coefficients, have irrationality measure zero. 
We also get the following dynamical fact.
For more details on these subjects see, for example, \cite{Hernandez} and \cite{Pollicott}. 

\begin{prop}
 If $\alpha$ is of bounded type then $\lambda^{+}(\alpha) < \infty$.
\end{prop}

The result known as the Thue-Siegel-Roth Theorem states that any algebraic irrational number has Irrationality Measure zero. 
There is no algebraic irrational number of degree $\geq 3$ for which its continued fraction expansion can be described completely, it is not even known if they are examples of bounded type. 
 
\begin{question}
 Do algebraic irrational numbers have finite upper Lyapunov Exponent for the Gauss Map?
\end{question}

Let $\alpha \in \mathbb{R}$ be irrational and ${p_n}/{q_n} \rightarrow \alpha$ be the approximation by convergents of the continued fractions. Suppose ${a_n}/{b_n} \rightarrow \alpha$ is an approximation such that \[\left| \alpha - \frac{a_n}{b_n} \right| \leq \left| \alpha - \frac{p_n}{q_n} \right|\] for all $n$.
Since the approximation by convergents is the best, we have that $b_n \geq q_n$ for all $n$. 
If the quantity \[ \limsup_{n \rightarrow \infty} \frac{1}{n} \log b_n \] is finite, then we can bound \[\lambda^{+} (\alpha) = \limsup_{n \rightarrow \infty} \frac{2}{n} \log q_n \leq \limsup_{n \rightarrow \infty} \frac{2}{n} \log b_n < \infty.\]
Hopefully, the generalization of the Gauss Map in the following sections will give such approximations for $\alpha$ algebraic irrational.

\section{The Multidimensional M\"onkemeyer Map}


In \cite{Panti} the multidimensional M\"onkemeyer Map is constructed as a generalization of the Farey Map and it gives rise to a dynamical construction of the corresponding generalized Question Mark Function, which conjugates the multidimensional M\"onkemeyer Map to a multidimensional Tent Map. 
We follow that construction. 
Consider the following matrices and simplex 
\[A = \mat{ccccccc}{1 & 0 & 0 & \cdots & 0 & 0 & 0 \\
                    1 & 0 & 0 & \cdots & 0 &-1 & 0 \\
                    0 & 1 & 0 & \cdots & 0 &-1 & 0 \\
                    \vdots & \vdots & \vdots & \ddots & \vdots & \vdots & \vdots \\
                    0 & 0 & 0 & \cdots & 0 &-1 & 0 \\
                    0 & 0 & 0 & \cdots & 1 &-1 & 0\\
                    0 & 0 & 0 & \cdots & 0 &-1 & 1 } \hspace*{2em}                    
  B = \mat{ccccccc}{0 & 0 & 0 & \cdots & 0 &-1 & 1 \\
                    1 & 0 & 0 & \cdots & 0 &-1 & 0 \\
                    0 & 1 & 0 & \cdots & 0 &-1 & 0 \\
                    \vdots & \vdots & \vdots & \ddots & \vdots & \vdots & \vdots \\
                    0 & 0 & 0 & \cdots & 0 &-1 & 0 \\
                    0 & 0 & 0 & \cdots & 1 &-1 & 0\\
                    1 & 0 & 0 & \cdots & 0 & 0 & 0 }\] 
\[V = \Mat{ccccccc}{0 & 1 & 1 & \cdots & 1 & 1 & 1 \\
                    0 & 0 & 1 & \cdots & 1 & 1 & 1 \\
                    0 & 0 & 0 & \cdots & 1 & 1 & 1 \\
                    \vdots & \vdots & \vdots & \ddots & \vdots & \vdots & \vdots \\
                    0 & 0 & 0 & \cdots & 0 & 1 & 1 \\
                    0 & 0 & 0 & \cdots & 0 & 0 & 1 \\
                    1 & 1 & 1 & \cdots & 1 & 1 & 1 }\]                
The multidimensional M\"onkemeyer Map is given as
\begin{equation*}
  M (\mathbf{v}) = 
   \left\lbrace
   \begin{array}{lcc}
   A \mathbf{v} & \text{ if } & \mathbf{v} \in A^{-1}V \\ 
   B \mathbf{v} & \text{ if } & \mathbf{v} \in B^{-1}V
   \end{array}
   \right.
\end{equation*}

We will consider again the first return map to the simplex $B^{-1} V$ and call that a \emph{Multidimensional Gauss Map}.
In the following sections we will explore the cases $n = 2, 3$, and leave the general case for later. 
In this setting, the group generated by the matrices $A$ and $B$ will play the role of $SL(2, \mathbb{Z})$.
We will refer to this group as the \emph{$n$-th M\"onkemeyer Group} $\mathcal{M}_n$.
From the previous section we have $\mathcal{M}_1 = SL(2, \mathbb{Z})$ and from the definition we have that $\mathcal{M}_n \subseteq SL(n+1, \mathbb{Z})$.

\section{The 2-dimensional case}

The multidimensional version of continued fractions as a dynamical system on a triangle has many different instances, see for example \cite{Beaver}, \cite{Dasaratha} and \cite{Garrity}. 

Consider the following matrices and simplex
\[ A = \mat{ccc}{1 & 0 & 0 \\ 1 & -1 & 0 \\ 0 & -1 & 1}, \hspace*{2em}
   B = \mat{ccc}{0 &-1 & 1 \\ 1 & -1 & 0 \\ 1 & 0 & 0}, \hspace*{2em}
   V = \Mat{ccc}{0 & 1 & 1 \\ 0 & 0 & 1 \\ 1 & 1 & 1 }.\]
Note that $A, B \in SL(3, \mathbb{Z})$ and the simplex represented by $V$ consists of all ordered pairs of numbers in the unit interval, i.e. \[ V = \left\lbrace \Mat{c}{x\\y\\1} \in \RP{2} : 0 \leq y \leq x  \leq 1 \right\rbrace .\]
The M\"onkemeyer Map $M: V \rightarrow V$ 
is given in homogeneous coordinates as 
\[ M (\mathbf{v}) = \left\lbrace \begin{array}{lcc} 
   A \mathbf{v} & \text{ if } & \mathbf{v} \in A^{-1}V \\ 
   B \mathbf{v} & \text{ if } & \mathbf{v} \in B^{-1}V \end{array} \right. ,\]
and when projected to $\mathbb{R}^2$ it takes the form 
\[ M(x, y) = \left\lbrace \begin{array}{lcc} 
        \left( \displaystyle\frac{x}{1-y}, \frac{x-y}{1-y} \right) & \text{ if } & x+y \leq 1 \\ \\
        \left( \displaystyle\frac{1-y}{x}, \frac{x-y}{x} \right) & \text{ if } & x+y \geq 1 \end{array} \right. .\] 
Note that $V = A^{-1}V \cup B^{-1}V$ and this sub-simplexes are the following triangles (See Figure \ref{triangle_2d}) \[ A^{-1} V = \Mat{ccc}{0&1&1 \\ 0&1&0 \\ 1&2&1} \hspace*{2em} B^{-1} V = \Mat{ccc}{1&1&1 \\ 1&1&0 \\ 1&2&1} .\]

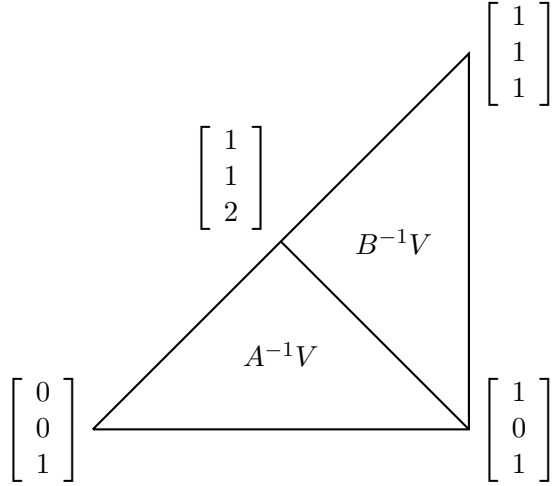
\begin{figure}
\begin{center}
 \begin{tikzpicture}[scale=5.0]
 
  \draw[thick] (0,0) --(1,0) --(1,1) --(0,0);
  \draw[thick] (1,0) --(0.5,0.5);
  
  \node[black, left] at (0,0) {$\Mat{c}{0 \\ 0 \\ 1}$};
  \node[black, right] at (1,1) {$\Mat{c}{1 \\ 1 \\ 1}$};
  \node[black, right] at (1,0) {$\Mat{c}{1 \\ 0 \\ 1}$};
  \node[black, above left] at (0.5,0.5) {$\Mat{c}{1 \\ 1 \\ 2}$};
  
  \node[black] at (0.5, 0.2) {$A^{-1}V$};
  \node[black] at (0.8, 0.5) {$B^{-1}V$};

 \end{tikzpicture}
 \caption{Partition of the base simplex for the 2-dimensional M\"onkemeyer Map}
 \label{triangle_2d}
\end{center}
\end{figure}

Consider again the first return map to the simplex $B^{-1}V$: 
For a point $\mathbf{v} \in V$ let $k_{\mathbf{v}}$ be the least non-negative integer such that $F^{k_{\mathbf{v}}} (\mathbf{v}) \in B^{-1}V$. 
The first return map is defined as
\[ G (\mathbf{v}) = \left\lbrace \begin{array}{lcc} 
   B A^{k_{\mathbf{v}}} \mathbf{v} & \text{ if } & \mathbf{v} \neq \mathbf{0} \\ 
   \mathbf{0} & \text{ if } & \mathbf{v} = \mathbf{0} \end{array} \right. ,
   \hspace*{2em} \text{where } \hspace*{2em} \mathbf{0} = \Mat{c}{0\\0\\1} .\]
Again, the matrix $A$ fixes $\mathbf{0}$ and any other point in $A^{-1}V$ is eventually mapped into $B^{-1}V$, so the map $G$ is defined piecewise on $V$.
We can retrieve the domains of definition of $G$ in the same way as in the 1-dimensional case, but we don't have the direct formula for the powers of $A$. 
Nonetheless, note that for all $n \in \mathbb{Z}$ 
\[ A^{2n} = \mat{ccc}{1 & 0 & 0 \\ 0 & 1 & 0 \\ -n & 0 & 1},\]
For $n \geq 1$ let $\mathbf{A}_n = B A^{2n -1}$ and $\mathbf{B}_n = B A^{2n-2}$. 
We have that \[ \mathbf{A}_n = B A^{-1} A^{2n} = 
\mat{ccc}{-n & 0 &1 \\ 0 & 1 & 0 \\ 1 & 0 & 0},\]
\[ \mathbf{B}_n = B A^{2(n-1)} = 
\mat{ccc}{1-n & -1 &1 \\ 1 & -1 & 0 \\ 1 & 0 & 0}.\]
The set of points that take $2n - 1$ iterations of $M$ to get to the simplex $B^{-1}V$ are in the simplex $\mathbb{A}_n= \mathbf{A}_n^{-1} V$ and the set of points that take $2n - 2$ iterations of $M$ to get to the simplex $B^{-1}V$ are in the simplex $\mathbb{B}_n= \mathbf{B}_n^{-1} V$. 
These simplexes are of the form: (See Figure \ref{gauss_2d})
\[ \mathbb{A}_n = \mathbf{A}_n^{-1} V = \mat{ccc}{0 & 0 & 1 \\ 0 & 1 & 0 \\ 1 & 0 & n} V = \Mat{ccc}{1 & 1 & 1 \\ 0 & 0 & 1 \\ n & n+1 & n+1}, \]
\[ \mathbb{B}_n = \mathbf{B}_n^{-1} V = \mat{ccc}{0 & 0 & 1 \\ 0 & -1 & 1 \\ 1 & -1 & n} V = \Mat{ccc}{1 & 1 & 1 \\ 1 & 1 & 0 \\ n & n+1 & n}. \]
 
\begin{figure}
\begin{center}
 \begin{tikzpicture}[scale=6.0]
  
  \draw[thick] (1, 0) --(0.5, 0.5) --(0.5, 0);
  \draw[thick] (0.5, 0) --(0.33333333, 0.33333333) --(0.33333333, 0);
  \draw[black!90!white, thick, dotted] (0.33333333, 0) --(0.25, 0.25) --(0.25, 0);
  \draw[black!80!white, thick, dotted] (0.25, 0) --(0.2, 0.2) --(0.2, 0);  
  \draw[black!70!white, thick, dotted] (0.2, 0) --(0.16666666, 0.16666666) --(0.16666666, 0);  
  \draw[black!60!white, thick, dotted] (0.16666666, 0) --(0.14285714, 0.14285714) --(0.14285714, 0);  
  \draw[black!50!white, thick, dotted] (0.14285714, 0) --(0.125, 0.125) --(0.125, 0);  
  \draw[black!40!white, thick, dotted] (0.125, 0) --(0.11111111, 0.11111111) --(0.11111111, 0);  
  \draw[black!25!white, thick, dotted] (0.11111111, 0) --(0.1, 0.1) --(0.1, 0);    
  \draw[black!10!white, thick, dotted] (0.1, 0) --(0.09090909, 0.09090909) --(0.09090909, 0);  
  
  \draw[thick] (0,0) --(1,0) --(1,1) --(0,0);
  
  \node[black, left] at (0,0) {$\Mat{c}{0 \\ 0 \\ 1}$};
  \node[black, right] at (1,1) {$\Mat{c}{1 \\ 1 \\ 1}$};
  
  \node[black, right] at (1,0) {$\Mat{c}{1 \\ 0 \\ 1}$};
  \node[black, below] at (0.53,0) {$\Mat{c}{1 \\ 0 \\ 2}$};
  \node[black, below] at (0.33333333,0) {$\Mat{c}{1 \\ 0 \\ 3}$};
  \node[black, above left] at (0.52,0.52) {$\Mat{c}{1 \\ 1 \\ 2}$};
  \node[black, above left] at (0.33333333,0.33333333) {$\Mat{c}{1 \\ 1 \\ 3}$};
  
  \node[black] at (0.65, 0.15) {$\mathbb{A}_1$};
  \node[black] at (0.4, 0.1) {$\mathbb{A}_2$};
  
  \node[black] at (0.8, 0.5) {$\mathbb{B}_1$};
  \node[black] at (0.43, 0.3) {$\mathbb{B}_2$};

 \end{tikzpicture}
 \caption{Partition of the base simplex for the 2-dimensional Gauss Map}
 \label{gauss_2d}
\end{center}
\end{figure}
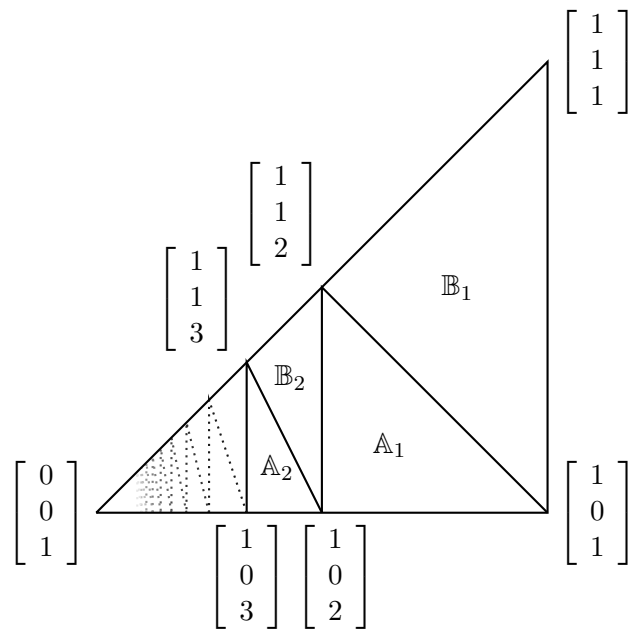

The map $G$ can be written as
\[ G (\mathbf{v}) = \left\lbrace \begin{array}{lcc} 
   \mathbf{A}_n \mathbf{v} & \text{ if } & \mathbf{v} \in \mathbb{A}_n \\ 
   \mathbf{B}_n \mathbf{v} & \text{ if } & \mathbf{v} \in \mathbb{B}_n \\ 
   \mathbf{0} & \text{ if } & \mathbf{v} = \mathbf{0} \end{array} \right. .\] 
and when projected to $\mathbb{R}^2$, the map $G$ takes the form 
\[ G (x, y) = \left\lbrace \begin{array}{lcc} 
   \left( \displaystyle\frac{1}{x} - n, \frac{y}{x}\right) & \text{ on } & \mathbb{A}_n \\ \\
   \left( \displaystyle\frac{1-y}{x} - n + 1, \frac{x-y}{x}\right) & \text{ on } & \mathbb{B}_n \\ \\
   (0,0) & \text{ if } & (x,y) = (0,0) \end{array} \right.\]
and it can be reduced to a formula more similar to the 1-dimensional Gauss Map
\[ G (x, y) = \left\lbrace \begin{array}{lcc} 
   \left( \displaystyle\frac{1}{x} \mod 1, \frac{y}{x}\right) & \text{ on } & \mathbb{A}_n \\ \\
   \left( \displaystyle\frac{1-y}{x} \mod 1, 1-\frac{y}{x}\right) & \text{ on } & \mathbb{B}_n \\ \\
   (0,0) & \text{ if } & (x,y) = (0,0) \end{array} \right. .\] 
   
\change

\begin{prop}
 Points of the triangle with rational coordinates are preimages of zero.
\end{prop}

\begin{proof}
 A point with rational coordinates is represented projectively by a point $\mathbf{r} \in V$ with integer entries, namely \[ \mathbf{r} = \Mat{c}{p \\ q \\ s}\] and it satisfies the inequalities $q \leq p \leq s$, the first comes from the ordering of the coordinates and the second from the coordinates being in $[0,1]$.
 The first inequality becomes equality when the point is in the diagonal edge and the second does it when the point is in the vertical edge.
 
 Suppose $\mathbf{r}$ is an interior point of the triangle, then the inequalities are strict. 
 Note that since $G$ preserves the triangle, it also preserves the inequalities. 
 Also, since the defining matrices $\mathbf{A}_n$ and $\mathbf{B}_n$ are of integer entries, they also keep the entries of $\mathbf{r}$ integer, in fact, \[ G \left( \Mat{c}{p \\ q \\ s} \right) = \left\lbrace \begin{array}{crl} 
   \Mat{c}{s - np \\ q \\ p} & \text{ on } & \mathbb{A}_n \\ 
   \Mat{c}{s - q + (1-n)p\\ p-q \\ p} & \text{ on } & \mathbb{B}_n \end{array} \right. . \]
 
 Now suppose that $G^{k} (\mathbf{r})$ is a interior point for all $k \geq 0$. 
 Let $\lbrace p_k \rbrace$, $\lbrace q_k \rbrace$ and $\lbrace s_k \rbrace$ be the integer sequences given by \[ G^{k} \left( \Mat{c}{p \\ q \\ s} \right) = \Mat{c}{p_k \\ q_k \\ s_k} \]
 by the previous statements we have that $s_0 > p_0 = s_1 > p_1 = s_2 > p_2 \cdots$, so the sequences are strictly decreasing, and therefore converging to zero, which is a contradiction.
 The lowest sequence $\lbrace q_k \rbrace$ would be the first to get to zero, meaning that $\mathbf{r}$ is mapped to the base of the triangle, where the sequences are no longer strictly decreasing. 
 In the base of the triangle the dynamics are, when projected to $\mathbb{R}^2$, $G(p/s, 0) = (s/p \mod 1, 0)$, i.e. $G$ is the 1-dimensional Gauss Map on the first coordinate.
 By a previous result, the rational points are mapped to zero. 
\end{proof}

The group $\mathcal{M}_2$ is generated by the matrices $A$ and $B$, but it is also generated by the two infinite families $\lbrace \mathbf{A}_n \rbrace $ and $\lbrace \mathbf{B}_n \rbrace$.
Note that the family $\lbrace \mathbf{A}_n \rbrace$ generates a subgroup homomorphic to $SL(2, \mathbb{Z})$, and the homomorphism is given by \[ \mat{cc}{a & b \\ c & d} \mapsto \mat{ccc}{a & 0 & b \\ 0 & 1 & 0 \\ c & 0 & d}.\]

As in the 1-dimensional case, the simplex $V$ represents all pairs of real numbers, up to integer parts and ordering. 
A property that would be desirable for the group $\mathcal{M}_2$ is that any point in $\RP{2}$ is related to a point in the simplex by the action of $\mathcal{M}_2$. 
As we have seen, we find a \emph{copy} of $SL(2, \mathbb{Z})$ in $\mathcal{M}_2$ and it comes with the translations \[\mat{ccc}{1 & 0 & -n \\ 0 & 1 & 0 \\ 0 & 0 & 1}\] which can take away the integral part, but only on the first coordinate.
There might be another subgroup taking care of the \emph{ordering} of the coordinates.

\begin{question}
 Is there an element in $\mathcal{M}_2$ which can permute the first and second coordinates of a point in $\RP{2}$?
\end{question}

In the 1-dimensional case the permutation part of the group $\mathcal{M}_1 = SL(2, \mathbb{Z})$ is trivial (since 1-tuples are trivially ordered) and the inversion and translations generate the whole group. 

\begin{conjecture}
 The group $\mathcal{M}_n$ might be generated by an inversion (for dynamics), translations (for no integer parts) and permutations (for ordered pairs).
\end{conjecture}

\begin{question}
 Is $\mathcal{M}_2$ the full $SL(3, \mathbb{Z})$?
\end{question}


Now consider the edges of the triangle, namely
\[
 \mathsf{A} = \Mat{cc}{0&1\\0&0\\1&1}, \mathsf{B} = \Mat{cc}{0&1\\0&1\\1&1} \text{ and } \mathsf{F} = \Mat{cc}{1&1\\0&1\\1&1}.
\]
The frontal edge $\mathsf{F}$ is contained in the piece $\mathbb{B}_1$ and is mapped in reversed order to the edge $\mathsf{B}$, i.e., \[G \Mat{c}{1\\x\\1} = \Mat{c}{1-x\\1-x\\1}, \] and when projected it takes the form $G(1,x) = (1-x, 1-x)$.

The base of the triangle, the edge $\mathsf{A}$, is contained in the pieces $\mathbb{A}_n$ and it is mapped to itself as a 1-dimensional Gauss Map, i.e., \[G \Mat{c}{x\\0\\1} = \Mat{c}{1-nx\\0\\x}, \] and when projected it takes the form $G(x,0) = (1/x \mod 1 , 0)$.

The diagonal edge $\mathsf{B}$, is contained in the pieces $\mathbb{B}_n$ and it is mapped to the edge $\mathsf{A}$ as a 1-dimensional Gauss Map, i.e., \[G \Mat{c}{x\\x\\1} = \Mat{c}{1-nx\\0\\x}, \] and when projected it takes the form $G(x,x) = (1/x \mod 1 , 0)$.

As we have seen, the dynamics on the \emph{base} of the triangle are the same as the 1-dimensional case, but there is more to it: 
Consider the full shift on the alphabet $\lbrace \mathbb{A}_n, \mathbb{B}_n : n \geq 1\rbrace$, this is conjugated to $G$ via itineraries. This conjugation is discussed in \cite{Panti} in terms of the map $M$. 

As for the 1-dimensional case, in order to pick a preferred itinerary for the preimages of edges and vertices of the triangle, we will take the convention 
\[ \mathbb{A}_n = \left\lbrace \Mat{c}{x \\ y \\ 1} \in \RP{2} : \frac{1}{n+1} < x \leq \frac{1}{n} \text{ and } 0 \leq y \leq 1 - n x \right\rbrace, \]
\[ \mathbb{B}_n = \left\lbrace \Mat{c}{x \\ y \\ 1} \in \RP{2} : \frac{1}{n+1} \leq x \leq \frac{1}{n} \text{ and } 1 - n x < y \leq x \right\rbrace. \]
Note that with this convention there are no defined preimages of the frontal edge $\mathsf{F}$.

Let $\mathbf{v} \in V$ be a non-rational point with itinerary $\mathbb{X}_{a_1}, \mathbb{X}_{a_2}, ...$, where $\mathbb{X}$ can be $\mathbb{A}$ or $\mathbb{B}$ and $a_k > 0$.
Again, for each $n \geq 1$ let $\mathbf{I}_{\mathbf{v},n} = \left( \mathbf{X}_{a_1}^{-1} \mathbf{X}_{a_2}^{-1} \cdots \mathbf{X}_{a_n}^{-1} \right)$, where $\mathbf{X}$ is $\mathbf{A}$ or $\mathbf{B}$ as indicated by the itinerary. 
The point $\mathbf{v}$ is in the simplex $\mathbb{I}_{n}(\mathbf{v}) = \mathbf{I}_{\mathbf{v},n} V$.
We have by construction that $\mathbb{I}_{n+1}(\mathbf{v}) \subset \mathbb{I}_{n}(\mathbf{v})$ and by the uniqueness of itineraries (See \cite{Panti}) that $\mathbf{v} = \bigcap_{n \geq 1} \mathbb{I}_{n}(\mathbf{v})$.
Again, we call the sequence $\mathbb{I}_{n}(\mathbf{v})$ the approximating simplexes of $\mathbf{v}$.

\begin{prop}
 A non-rational point $\mathbf{v} \in V$ has an itinerary consisting only of symbols of type $\mathbb{A}_n$ if and only if it is a point in the base of the triangle.
 Moreover, if the itinerary is $\mathbb{A}_{a_1}, \mathbb{A}_{a_2}, ...$ then the irrational number on the first coordinate is given in continued fractions as $[0: a_1, a_2, ...]$.
\end{prop}

\begin{proof}
 The second part of the statement follows directly from the previous discussion on the dynamics restricted to the base of the triangle.
 For the first part, note that if a point is in the base of the triangle then it is in the \emph{base} of some sub-simplex $\mathbb{A}_n$ and it is mapped again to the base of the triangle, and therefore its itinerary consists only of $\mathbb{A}_n$ symbols.
 
 Now suppose $\mathbf{v}$ has the itinerary $\mathbb{A}_{a_1}, \mathbb{A}_{a_2}, ...$, then by previous results the approximating simplexes of $\mathbf{v}$ are of the form \[ \mathbb{I}_{n}(\mathbf{v}) = \mathbf{A}_{a_1}^{-1} \mathbf{A}_{a_2}^{-1} \cdots \mathbf{A}_{a_n}^{-1} V  = \mat{ccc}{p_{n-1} & 0 & p_n \\ 0 & 1 & 0 \\ q_{n-1} & 0 & q_n} V = \Mat{ccc}{p_{n-1} & p_{n-1} + p_{n} & p_{n-1} + p_{n} \\ 0 & 0 & 1 \\ q_n & q_{n-1} + q_n & q_{n-1} + q_n},\]
 where $p_n/q_n$ are the convergents of the irrational number $[0: a_1, a_2, ...]$.
 This simplexes projected to $\mathbb{R}^2$ are triangles with vertices 
 \[ \left( \frac{p_n}{q_n}, 0 \right), \hspace*{1em} \left( \frac{p_{n-1} + p_{n}}{q_{n-1} + q_n}, 0 \right), \hspace*{1em} \text{and} \hspace*{1em} \left( \frac{p_{n-1} + p_{n}}{q_{n-1} + q_n}, \frac{1}{q_{n-1} + q_n} \right).\]
 When projected to $\mathbb{R}^2$, the point $\mathbf{v} = (x, y)$ is in all of these triangles. 
 Coordinate-wise, $x$ is approximated by the convergents $p_n/q_n$ and $y \leq (q_{n-1} + q_n)^{-1} \rightarrow 0$.
 Therefore, $(x, y) = ([0: a_1, a_2, ...], 0)$.
\end{proof}
 
\begin{coro}
 Points with an itinerary with a tail consisting only of symbols of type $\mathbb{A}_n$ are preimages of the edges of the triangle.
\end{coro}

This corollary follows from the conjugation to symbolic dynamics.  
We can say that the 1-dimensional Gauss Map \emph{lives} inside the 2-dimensional one, not just as the restriction to the base edge of the triangle but also as a sub-shift in the symbolic dynamics.

\begin{question}
 Is there an arithmetic/algebraic relation between the coordinates of a point in the preimages of the edges?
\end{question}

Let $\alpha$ be an irrational number.
Consider the points in the triangle of the form $(\alpha, y)$ that are eventually mapped to an edge.
As we have seen, the endpoints of the preimages of edges are rational points, so these are segments of straight lines with rational slope and we can express the $y$ coordinate of such points as $y = \frac{p}{q} \alpha + \frac{r}{s}$, for some $p/q, r/s  \in \mathbb{Q}$.

\begin{conjecture}
 A non-rational point $(x, y)$ is eventually mapped to an edge if there are some rational numbers $p/q, r/s$ such that $y = \frac{p}{q} x + \frac{r}{s}$
\end{conjecture}

As we have seen, in the base of the triangle we have the 1-dimensional Gauss Map, and for these points the \emph{bases} of the approximating simplexes are the corresponding 1-simplexes coming from the convergents.

Let $\mathbf{v}$ be a non rational point and \[\mathbb{I}_{n}(\mathbf{v}) = \Mat{ccc}{p_{n,1} & p_{n,2} & p_{n,3} \\ q_{n,1} & q_{n,2} & q_{n,3} \\ r_{n,1} & r_{n,2} & r_{n,3}} \] an approximating simplex. Consider the tetrahedron in $\mathbb{R}^3$ with one vertex in the origin and the other three in the points $(p_{n,i}, q_{n,i}, r_{n,i})$, for $i=1,2,3$. 
Suppose there is a point of $\mathbb{Z}^3$ inside this tetrahedron, that would mean there is a straight line passing through such point, and that line would correspond to a rational point closer to $\mathbf{v}$ than any vertex of the approximating simplex.
Moreover, being inside this tetrahedron would imply that, when projected to $\mathbb{R}^2$, the common denominator of the coordinates of this point would be smaller than the common denominator of the vertices, which would imply we can approximate $\mathbf{v}$ better with a lower common denominator.

\begin{prop}[Best approximations]
 There are no integral points inside an approximating simplex other than its vertices.
\end{prop}

\begin{proof}
 Note that the matrices that define the simplexes $\mathbb{I}_{n}(\mathbf{v})$ are in $SL(3, \mathbb{Z})$.
 Suppose there is a point with integer coordinates inside the associated tetrahedron, namely \[ \Mat{c}{i\\j\\k}. \]
 Since these tetrahedra are convex subsets of $\mathbb{R}^3$ there exist three numbers $\varepsilon_1, \varepsilon_2, \varepsilon_3 \in [0,1]$ such that $\varepsilon_1 + \varepsilon_2 + \varepsilon_3 = 1$ and \[\Mat{ccc}{p_{n,1} & p_{n,2} & p_{n,3} \\ q_{n,1} & q_{n,2} & q_{n,3} \\ r_{n,1} & r_{n,2} & r_{n,3}} \Mat{c}{\varepsilon_1 \\ \varepsilon_2 \\ \varepsilon_3} = \Mat{c}{i\\j\\k}.\]
 But $\mathbb{I}_{n}(\mathbf{v})^{-1}$ exists and it also has integer entries so the equation \[ \mathbb{I}_{n}(\mathbf{v})^{-1} \Mat{c}{i\\j\\k}  = \Mat{c}{\varepsilon_1 \\ \varepsilon_2 \\ \varepsilon_3}\]
 forces the numbers $\varepsilon_i$ to be integers. 
 The only options are \[ \Mat{c}{\varepsilon_1 \\ \varepsilon_2 \\ \varepsilon_3} = \Mat{c}{1\\0\\0} \hspace*{2em} \text{or} \hspace*{2em} \Mat{c}{\varepsilon_1 \\ \varepsilon_2 \\ \varepsilon_3} = \Mat{c}{0\\1\\0} \hspace*{2em} \text{or} \hspace*{2em} \Mat{c}{\varepsilon_1 \\ \varepsilon_2 \\ \varepsilon_3} = \Mat{c}{0\\0\\1}.\]
 So the only integral points inside an approximating simplex are its vertices.
\end{proof}

%
%
\change

Consider a periodic point, $\mathbf{v} = G^k (\mathbf{v})$, with itinerary $\overline{\mathbb{X}_{a_1}, \mathbb{X}_{a_2}, ... \mathbb{X}_{a_k}}$.
The periodic point equation implies that $\mathbf{v}$ is an eigenvector for the matrix $ \left( \mathbf{X}_{a_k} \cdots \mathbf{X}_{a_2}  \mathbf{X}_{a_1} \right) = \mathbf{I}_{\mathbf{v}, n}^{-1}$. 
Let $\lambda$ be the associated eigenvalue, by definition $\lambda$ is an algebraic number of degree at most 3, and from the system of equations
\[ \left( \mathbf{I}_{\mathbf{v}, n}^{-1} - \lambda \mat{ccc}{1&0&0\\0&1&0\\0&0&1}\right) \mat{c}{x \\ y \\ 1} = \mat{c}{0\\0\\0},\] 
we get that the entries $x$ and $y$ are in the number field $\mathbb{Q}(\lambda)$.

For example, consider the fixed points of $G$. 
The eigenpolynomial of the matrix $\mathbf{A}_n$ is $(1-\lambda)(\lambda^2 + n \lambda - 1)$.
The number $\lambda_0 = [0: \overline{n}]$ is a root of the second factor, and the associated eigenvector is \[ \Mat{c}{\lambda_0 \\ 0 \\ 1}. \] 

The eigenpolynomial of the matrix $\mathbf{B}_n$ is $\lambda^3 + n \lambda^2 + (n-1) \lambda - 1$ and it has one real root  $\lambda \in (0, 1)$. The eigenvector is determined by the system of equations
\[\mat{ccc}{1-n-\lambda & -1 & 1 \\ 1 & -1-\lambda & 0 \\ 1 & 0 & -\lambda} \mat{c}{x\\y\\1} = \mat{c}{0\\0\\0},\]
\[\mat{ccc}{1 & 0 & -\lambda \\ 0 & 1 & -\frac{\lambda}{1+\lambda} \\ 0 & 0 & \lambda^3 + n \lambda^2 + (n-1) \lambda - 1} \mat{c}{x\\y\\1} = \mat{c}{0\\0\\0},\]
and we get that $x = \lambda$ and $y = \frac{\lambda}{1+\lambda}$.
For example, the point with itinerary $[\overline{\mathbb{B}_1}]$ is 
\[ \Mat{c}{\lambda \\ \frac{\lambda}{\lambda+1} \\ 1} \hspace*{2em} \text{with} \hspace*{2em} \lambda = \frac{1}{3} \left( \sqrt[3]{\frac{25 + 3 \sqrt{69}}{2}} + \sqrt[3]{\frac{25 - 3 \sqrt{69}}{2}} -1 \right)\]
 
\begin{conjecture}
A point in the triangle is eventually periodic if both coordinates are in the same cubic number field.
\end{conjecture}
 
\textbf{Example 1.}
Consider the number $\alpha = \sqrt[3]{2} - 1$ with minimal polynomial $\alpha^3 + 3 \alpha^2 + 3 \alpha - 1 = 0$, and the point \[\mathbf{w} = \Mat{c}{\alpha \\ \alpha^2 \\ 1} \approx \Mat{c}{0.25992 \\ 0.06755 \\ 1}.\]
We have that $1/4 < \alpha < 1/3$ and $\alpha^2 < 1-3\alpha$ so the point $\mathbf{w}$ is in the piece $\mathbb{A}_3$.
The next point is 
\[G(\mathbf{w}) = \mathbf{A}_3 \mathbf{w} 
                = \Mat{c}{1-3\alpha \\ \alpha^2 \\ \alpha} 
                \approx \Mat{c}{0.84732 \\ 0.25992 \\ 1}. \]
This point is in the piece $\mathbb{B}_1$, and we have
\[G^2(\mathbf{w}) = \mathbf{B}_1 \Mat{c}{1 - 3\alpha\\ \alpha^2 \\ \alpha} 
                  = \Mat{c}{\alpha - \alpha^2 \\ 1 - 3 \alpha - \alpha^2 \\ 1 - 3 \alpha} 
                  \approx \Mat{c}{0.87343 \\ 0.69324 \\ 1}.\]
This point is again in the piece $\mathbb{B}_1$,
\[G^3(\mathbf{w}) = \mathbf{B}_1 \Mat{c}{\alpha - \alpha^2 \\ 1 - 3 \alpha - \alpha^2 \\ 1 - 3 \alpha} 
                  = \Mat{c}{\alpha^2 \\ 4 \alpha - 1 \\ \alpha - \alpha^2}
                  \approx \Mat{c}{0.351207 \\ 0.206299 \\ 1}.\]
The next symbol is $\mathbb{A}_2$ and we have
\[G^4(\mathbf{w}) = \mathbf{A}_2 \Mat{c}{\alpha^2 \\ 4 \alpha - 1 \\ \alpha - \alpha^2}
                  = \Mat{c}{\alpha - 3 \alpha^2 \\ 4 \alpha - 1 \\ \alpha^2}
                  \approx \Mat{c}{0.847322 \\ 0.587401 \\ 1}.\]
This point is again in the piece $\mathbb{B}_1$,
\[G^5(\mathbf{w}) = \mathbf{B}_1 \Mat{c}{\alpha - 3 \alpha^2 \\ 4 \alpha - 1 \\  \alpha^2}
                  = \Mat{c}{\alpha^2 - 4 \alpha + 1 \\ 1 - 3\alpha - 3\alpha^2 \\ \alpha - 3 \alpha^2}
                  \approx \Mat{c}{0.486945 \\ 0.306756 \\ 1}.\]
The next symbol is $\mathbb{B}_2$,
\[G^6(\mathbf{w}) = \mathbf{B}_2 \Mat{c}{\alpha^2 - 4 \alpha + 1 \\ 1 -3 \alpha^2 -3 \alpha \\ \alpha - 3 \alpha^2}
                  = \Mat{c}{-\alpha^2 + 8 \alpha - 2 \\ 4 \alpha^2 - \alpha \\ \alpha^2 -4 \alpha + 1}
                  \approx \Mat{c}{0.423661 \\ 0.370039 \\ 1}.\]
This point is again in the piece $\mathbb{B}_2$ and with a substitution of the minimal polynomial we get,           
\[G^7(\mathbf{w}) = \mathbf{B}_2 \Mat{c}{-2 \alpha^2 -7 \alpha +2 \\ 4\alpha -1 \\ \alpha^2 +4 \alpha -1}
                  = \Mat{c}{ 3 \alpha^2 +7 \alpha -2 \\ -2 \alpha^2 -11 \alpha +3 \\ -2 \alpha^2 -7 \alpha +2} 
                  \approx \Mat{c}{0.486944 \\ 0.126567 \\ 1}.\]       
This point is again in the piece $\mathbb{B}_2$,           
\[G^8(\mathbf{w}) = \mathbf{B}_2 \Mat{c}{ 3 \alpha^2 +7 \alpha -2 \\ -2 \alpha^2 -11 \alpha +3 \\ -2 \alpha^2 -7 \alpha +2} 
                  = \Mat{c}{-3 \alpha^2 -3 \alpha +1 \\ 5 \alpha^2 +18 \alpha -5 \\ 3 \alpha^2 +7 \alpha -2} 
                  \approx \Mat{c}{0.793701 \\ 0.740079 \\ 1}.\]       
The next symbol is $\mathbb{B}_1$ and with a substitution of the minimal polynomial we get
\[G^9(\mathbf{w}) = \mathbf{B}_1 \Mat{c}{ \alpha^2 \\ -5 \alpha^2 -10 \alpha +3 \\ -2 \alpha^2 -3 \alpha +1} 
                  = \Mat{c}{ 3 \alpha^2 +7 \alpha -2 \\ 6 \alpha^2 +10 \alpha -3 \\ \alpha^2}
                  \approx \Mat{c}{0.327480 \\ 0.067558 \\ 1}.\]   
This point is in the piece $\mathbb{B}_3$,
\[G^{10}(\mathbf{w}) = \mathbf{B}_3 \Mat{c}{ 3 \alpha^2 +7 \alpha -2 \\ 6 \alpha^2 +10 \alpha -3 \\ \alpha^2}
                  = \Mat{c}{-11 \alpha^2 -24 \alpha +7 \\ -3 \alpha^2 -3 \alpha +1 \\ 3 \alpha^2 +7 \alpha -2}
                  \approx \Mat{c}{0.847322 \\ 0.793701 \\ 1}.\]   
The next symbol is $\mathbb{B}_1$ and with a substitution of the minimal polynomial we get
\[G^{11}(\mathbf{w}) = \mathbf{B}_1 \Mat{c}{ 7 \alpha^2 + 10 \alpha -3\\ \alpha^2 \\ -2 \alpha^2 -3 \alpha +1} 
                  = \Mat{c}{ -3 \alpha^2 -3\alpha +1 \\ 6 \alpha^2 +10 \alpha -3 \\ 7\alpha^2 +10\alpha -3}
                  \approx \Mat{c}{0.243472 \\ 0.0632835 \\ 1}.\]   
The next symbol is $\mathbb{B}_4$ and with a substitution of the minimal polynomial we get
\[G^{12}(\mathbf{w}) = \mathbf{B}_4 \Mat{c}{ \alpha \\ \alpha^2 \\ \alpha^2 +1} 
                  = \Mat{c}{ -3 \alpha  + 1 \\ - \alpha^2 + \alpha \\ \alpha}
                  \approx \Mat{c}{0.847322 \\ 0.740079 \\ 1}.\]   
This point is in the piece $\mathbb{B}_1$,
\[G^{13}(\mathbf{w}) = \mathbf{B}_1 \Mat{c}{ -3 \alpha  + 1 \\ - \alpha^2 + \alpha \\ \alpha}
                  = \Mat{c}{ \alpha^2 \\ \alpha^2 -4 \alpha +1 \\ - 3 \alpha + 1}
                  \approx \Mat{c}{0.306755 \\ 0.126567 \\ 1}.\] 
The next symbol is $\mathbb{B}_3$ and we get $G^{14}(\mathbf{w}) = G^{4}(\mathbf{w})$:
\[G^{14}(\mathbf{w}) = \mathbf{B}_3 \Mat{c}{ \alpha^2 \\ \alpha^2 -4 \alpha +1 \\ - 3 \alpha + 1}
                  = \Mat{c}{ \alpha -3 \alpha^2  \\ 4\alpha -1 \\ \alpha^2}
                  \approx \Mat{c}{0.847322 \\ 0.587401 \\ 1}.\]     
So the itinerary of the point $\mathbf{w}$ is $[\mathbb{A}_{3}, \mathbb{B}_{1}, \mathbb{B}_{1}, \mathbb{A}_{2}, \overline{\mathbb{B}_{1}, \mathbb{B}_{2}, \mathbb{B}_{2}, \mathbb{B}_{2}, \mathbb{B}_{1}, \mathbb{B}_{3}, \mathbb{B}_{1}, \mathbb{B}_{4}, \mathbb{B}_{1}, \mathbb{B}_{3}}]$.

The following example shows that being in the same cubic number field does not imply periodicity.

\textbf{Example 2.} Consider again $\alpha = \sqrt[3]{2} - 1$ and the point \[ \mathbf{w} = \Mat{c}{2\alpha \\ \alpha \\ 2 }.\]
The itinerary of this point begins with $[\mathbb{A}_3, \mathbb{B}_1, \mathbb{A}_1, \mathbb{B}_1, ...]$.
\[ G(\mathbf{w}) = \mathbf{A}_3 \Mat{c}{2\alpha \\ \alpha \\ 2 }
                 = \Mat{c}{2 - 6\alpha \\ \alpha \\ 2 }, \hspace*{2em}
   G^2(\mathbf{w}) = \mathbf{B}_1 \Mat{c}{2 - 6\alpha \\ \alpha \\ 2 }
                 = \Mat{c}{\alpha \\ 2 -7\alpha \\ 2 -6\alpha },
\]
\[ G^3(\mathbf{w}) = \mathbf{A}_1 \Mat{c}{\alpha \\ 2 -7\alpha \\ 2 -6\alpha }
                 = \Mat{c}{2 - 7\alpha \\ 2 - 7\alpha \\ \alpha }, \hspace*{2em}
   G^4(\mathbf{w}) = \mathbf{B}_1 \Mat{c}{2 - 7\alpha \\ 2 - 7\alpha \\ \alpha }
                 = \Mat{c}{8\alpha -2\\0 \\ 2 - 7\alpha },
\]
and from that, the itinerary follows the continued fraction expansion of $(8 \alpha -2)/(2 - 7\alpha)$, which is not periodic.

\begin{conjecture}
 To be a periodic point, the numbers $x$ and $y$ must be in the same cubic number field and also be linear independent over the rationals.
\end{conjecture}

These examples suggest that the dynamics of $G$ \emph{collapse} any rational dependence on the coordinates. 
For a quadratic irrational $\alpha$, the corresponding number field can be seen as a vector space over $\mathbb{Q}$ generated by $\alpha$ and 1, so any pair of numbers would be rational dependent, and as a previous conjecture states, the corresponding point would be mapped to the edge of the triangle, where the 1-dimensional dynamics takes place, eventually becoming periodic. 
As in the second example, the numbers are cubic irrationals but they are rational dependent. 
The dynamics \emph{collapse} this dependency and the point is mapped to the edge, but in this case it can not be periodic.
Periodicity in dimension $n$ implies that the coordinates are in the same number field of degree at most $n$. 
An algebraic number of higher degree does not \emph{fit} as an eigenvalue of a matrix in $\mathcal{M}_n$.

\begin{question}
 Do all cubic irrational numbers in $[0,1]$ appear as a coordinate of a periodic point of $G$?
\end{question}

On \cite{Murru} each cubic irrational number $\alpha$ is realized as a periodic point of a dynamical system on the set of pairs of numbers in the cubic number field $\mathbb{Q}(\alpha)$.

\begin{question}
 Are the cubic number fields so diverse that a single dynamical system can not give an affirmative answer to the previous question? 
\end{question}

\change

\begin{question}
 Is the rate of growth of the denominators of the approximating simplexes related to the rate of approximation?
\end{question}

Consider the case of non rational points in the triangle and write the sequence of approximating simplexes as \[ \mathbb{I}_{n}(\mathbf{v}) = \Mat{ccc}{p_{n,1} & p_{n,2} & p_{n,3} \\ q_{n,1} & q_{n,2} & q_{n,3} \\ r_{n,1} & r_{n,2} & r_{n,3}} \hspace*{2em} \text{ where } \hspace*{2em} \mathbf{v} = \Mat{c}{x \\ y \\ 1}\] 
Suppose there is a $m>0$ such that for all $n$ \[ \left| x - \frac{p_{n,i}}{r_{n,i}} \right| + \left| y - \frac{q_{n,i}}{r_{n,i}} \right| < \frac{1}{r_{n,i}^m},\]
we can get an upper bound for the number $m$ as \[\gamma_i(n) := \frac{-\log\left( \left| x - \frac{p_{n,i}}{r_{n,i}} \right| + \left| y - \frac{q_{n,i}}{r_{n,i}} \right| \right)}{\log(r_{n,i})} > m.\]
In the Figure \ref{rate_convergence2} are the values of $\gamma_i(n)$ for the first 50 approximating simplexes of the fixed point $[\overline{\mathbb{B}_1}]$. 
Following the intuition from the 1-dimensional case, we would expect this point to be approximated the \emph{slowest}, and as the Figure \ref{rate_convergence2} suggests, the value of the bound $m$ would be between 1 and 2.

\begin{figure}
 \begin{center}
  \begin{tikzpicture}[scale=2.0]
   \draw[blue!30!white, dashed] (0,2) --(5,2);
   \draw[blue!30!white, dashed] (0,1) --(5,1);
   \draw[->, thick, black] (0,0) --(5,0);
   \draw[->, thick, black] (0,0) --(0,3);
   \node[black] at (2.5,-0.2) {$n$};
   \node[blue!30!white] at (-0.1,2) {$2$};
   \node[blue!30!white] at (-0.1,1) {$1$};
   \node[blue] at (-0.4,1.8) {$\gamma_1(n)$};
   \node[green!80!black] at (-0.4,1.5) {$\gamma_2(n)$};
   \node[red!90!blue] at (-0.4,1.2) {$\gamma_3(n)$};
   
   \input{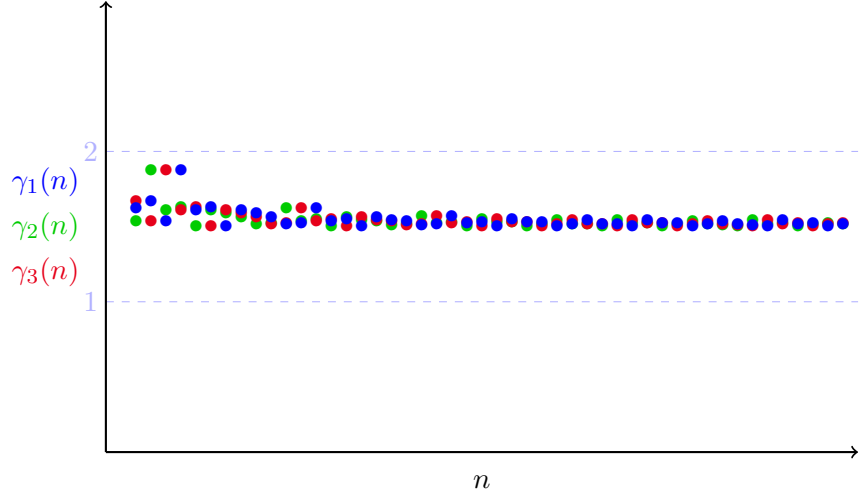}
  \end{tikzpicture}
 \end{center}
 \caption{Graph of the functions $\gamma_i(n)$ for the fixed point $[\overline{\mathbb{B}_1}]$.}
 \label{rate_convergence2}
\end{figure}

\change




\begin{question}
 Is there a relation between the upper Lyapunov Exponent for the map $G$ and the exponential growth rate of the denominators in the approximating simplexes?
\end{question}

\section{The 3-dimensional case}

Once again, consider the matrices and simplex
\[ A = \mat{cccc}{1 & 0 & 0 & 0 \\ 1 & 0 & -1 & 0 \\ 0 & 1 & -1 & 0 \\ 0 & 0 & -1 & 1}, \hspace*{2em}
   B = \mat{cccc}{0 & 0 & -1 & 1 \\ 1 & 0 & -1 & 0 \\ 0 & 1 & -1 & 0 \\ 1 & 0 & 0 & 0}, \hspace*{2em}
   V = \Mat{cccc}{0 & 1 & 1 & 1\\ 0 & 0 & 1 & 1 \\ 0 & 0 & 0 & 1 \\ 1 & 1 & 1 & 1 }.\]
Note that $A, B \in SL(4, \mathbb{Z})$ and the simplex represented by $V$ consists of all ordered triplets of numbers in the unit interval. 
The M\"onkemeyer Map $M: V \rightarrow V$ 
is given in homogeneous coordinates as 
\[ M (\mathbf{v}) = \left\lbrace \begin{array}{lcc} 
   A \mathbf{v} & \text{ if } & \mathbf{v} \in A^{-1}V \\ 
   B \mathbf{v} & \text{ if } & \mathbf{v} \in B^{-1}V \end{array} \right. ,\]
and when projected to $\mathbb{R}^3$ it takes the form 
\[ M(x, y, z) = \left\lbrace \begin{array}{lcc} 
        \left( \displaystyle\frac{x}{1-z}, \frac{x-z}{1-z} , \frac{y-z}{1-z} \right) & \text{ if } & x+z \leq 1 \\ \\
        \left( \displaystyle\frac{1-z}{x}, \frac{x-z}{x}, \frac{y-z}{x} \right) & \text{ if } & x+z \geq 1 \end{array} \right. .\] 
Note that $V = A^{-1}V \cup B^{-1}V$ and this sub-simplexes are the following tetrahedra (See Figure \ref{tetrahedron_V})
\[ A^{-1} V = \Mat{cccc}{0&1&1&1 \\ 0&1&0&1 \\ 0&1&0&0 \\ 1&2&1&1} \hspace*{2em} 
   B^{-1} V = \Mat{cccc}{1&1&1&1 \\ 1&1&0&1 \\ 1&1&0&0 \\ 1&2&1&1} .\]

\begin{figure}
\begin{center}
 \begin{tikzpicture}[scale=5.0]
 
 
  \draw[dashed] (0,0) -- (1,0);
  \draw[thick] (0,0) --(0.5,-0.25) --(1,0) --(1,0.5) --(0,0);
  \draw[thick] (0.5,-0.25) --(1,0.5);
  
  \node[black, left] at (0,0) {$\Mat{c}{0 \\ 0 \\ 0 \\ 1}$};
  \node[black, right] at (1,0.6) {$\Mat{c}{1 \\ 1 \\ 1 \\ 1}$};
  \node[black, right] at (1,0) {$\Mat{c}{1 \\ 1 \\ 0 \\ 1}$};
  \node[black, below] at (0.5,-0.25) {$\Mat{c}{1 \\ 0 \\ 0 \\ 1}$};
  

 \end{tikzpicture}
 
 \begin{tikzpicture}[scale=5.0]
 
  \draw[dashed] (0,0) -- (1,0);
  \draw[thick] (0,0) --(0.5,-0.25) --(1,0) --(0.5,0.25) --(0,0);
  \draw[thick] (0.5,0.25) --(0.5,-0.25);
  \draw[dotted] (0.5,-0.25) --(1,0.5);
  \draw[dotted] (1,0) --(1,0.5) --(0.5,0.25);
  
  
  \node[black, above right] at (0,-0.25) {$A^{-1}V$};
  
  \draw[dotted] (1.5,0) --(2.5,0);
  \draw[dashed] (2.5,0) --(2,0.25);
  \draw[thick] (2,-0.25) --(2.5,0) --(2.5,0.5) --(2,0.25) --(2,-0.25) --(2.5,0.5);
  \draw[dotted] (2,-0.25)-- (1.5, 0) --(2, 0.25); 
  
  
  \node[black, above left] at (2.5,-0.25)  {$B^{-1}V$};

 \end{tikzpicture}
 
 \caption{Partition of the base simplex for the 3-dimensional M\"onkemeyer Map}
 \label{tetrahedron_V}
\end{center}
\end{figure}
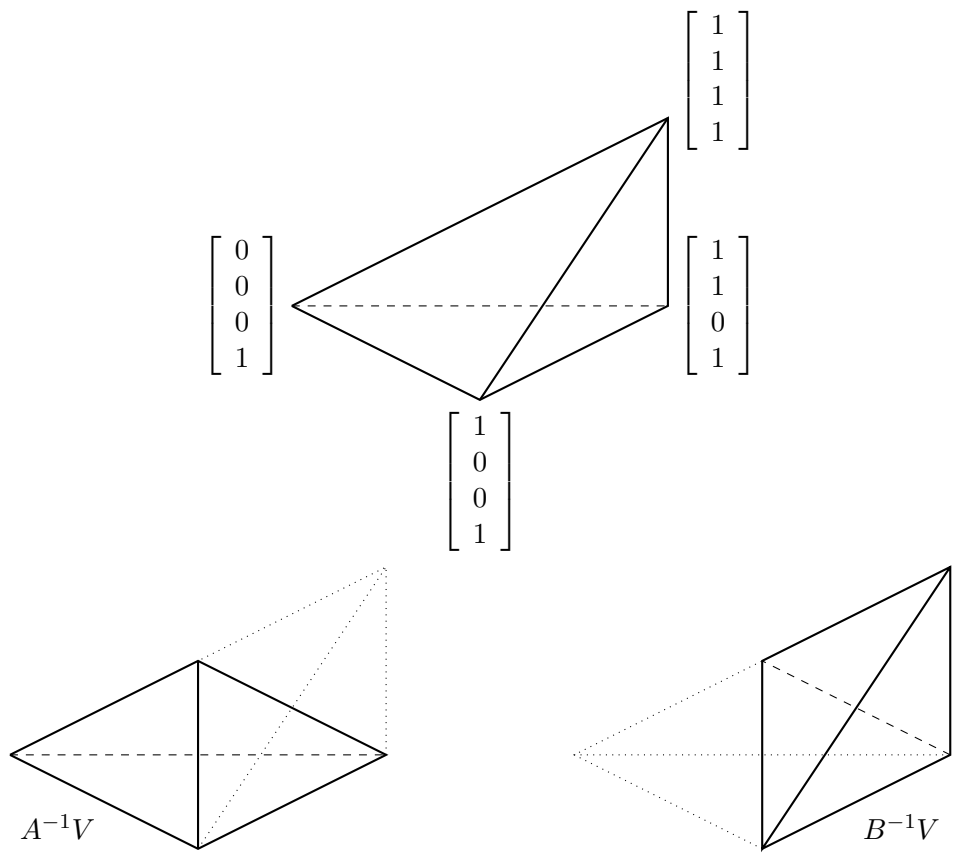
   
Again, the matrix $A$ fixes the point \[\mathbf{0} = \Mat{c}{0\\0\\0\\1}\] and any other point in $A^{-1}V$ is eventually mapped into $B^{-1}V$, so the first return map is defined piecewise on $V$.
In this case we have that for all $n \in \mathbb{Z}$ 
\[ A^{3n} = \mat{cccc}{1 & 0 & 0 & 0\\ 0 & 1 & 0 & 0 \\ 0 & 0 & 1 & 0 \\ -n & 0 & 0 & 1},\]
For $n \geq 1$ let $\mathbf{A}_n = B A^{3n -1}$, $\mathbf{B}_n = B A^{3n-2}$ and $\mathbf{C}_n =  B A^{3n-3}$. 
We have  
\[ \mathbf{A}_n = B A^{-1} A^{3n} = \mat{cccc}{-n & 0 & 0 & 1 \\ 0 & 1 & 0 & 0 \\ 0 & 0 & 1 & 0\\ 1 & 0 & 0 & 0}, \hspace*{2em} \mathbf{B}_n = B A^{-2} A^{3n} = \mat{cccc}{1-n & -1 & 0 & 1 \\ 1 & -1 & 1 & 0 \\ 1 & -1 & 0 & 0 \\ 1 & 0 & 0 & 0},\]
\[ \mathbf{C}_n = B A^{3(n-1)} = \mat{cccc}{1-n & 0 & -1 & 1 \\ 1 & 0 & -1 & 0 \\ 0 & 1 & -1 & 0 \\ 1 & 0 & 0 & 0}.\]
Consider again the simplexes $\mathbb{A}_n= \mathbf{A}_n^{-1} V$, $\mathbb{B}_n= \mathbf{B}_n^{-1} V$ and $\mathbb{C}_n= \mathbf{C}_n^{-1} V$. 
These simplexes are (See Figure \ref{gauss_3d})
\[ \mathbb{A}_n = \Mat{cccc}{1 & 1 & 1 & 1 \\ 0 & 0 & 1 & 1 \\ 0 & 0 & 0 & 1 \\ n & n+1 & n+1 & n+1}, \hspace*{2em}
   \mathbb{B}_n = \Mat{cccc}{1 & 1 & 1 & 1 \\ 1 & 1 & 1 & 0 \\ 0 & 0 & 1 & 0 \\ n & n+1 & n+1 & n}. \]
\[ \mathbb{C}_n = \Mat{cccc}{1 & 1 & 1 & 1 \\ 1 & 1 & 0 & 1 \\ 1 & 1 & 0 & 0 \\ n & n+1 & n & n}. \]

The map $G$ can be written again as
\[ G (\mathbf{v}) = \left\lbrace \begin{array}{lcc} 
   \mathbf{A}_n \mathbf{v} & \text{ if } & \mathbf{v} \in \mathbb{A}_n \\ 
   \mathbf{B}_n \mathbf{v} & \text{ if } & \mathbf{v} \in \mathbb{B}_n \\ 
   \mathbf{C}_n \mathbf{v} & \text{ if } & \mathbf{v} \in \mathbb{C}_n \\ 
   \mathbf{0} & \text{ if } & \mathbf{v} = \mathbf{0} \end{array} \right. .\] 
and when projected to $\mathbb{R}^3$, the map $G$ takes the form 
\[ G (x, y, z) = \left\lbrace \begin{array}{lcc} 
   \left( \displaystyle\frac{1}{x} - n, \frac{y}{x}, \frac{z}{x}\right) & \text{ on } & \mathbb{A}_n \\ \\
   \left( \displaystyle\frac{1-y}{x} - n + 1, \frac{x-y+z}{x}, \frac{x-y}{x}\right) & \text{ on } & \mathbb{B}_n \\ \\
   \left( \displaystyle\frac{1-z}{x} - n + 1, \frac{x-z}{x}, \frac{y-z}{x}\right) & \text{ on } & \mathbb{C}_n \\ \\
   (0,0) & \text{ if } & (x,y) = (0,0) \end{array} \right.\]
(The formulas can be worked to be independent of $n$.)

\begin{figure}
\begin{center}
 \begin{tikzpicture}[scale=5.0]
  \draw[fill=blue!10!white, draw=none] (0.33333333, 0.16666666) --(0.4, 0) --(0.16666666, -0.08333333) --(0.33333333, 0.16666666);
  \draw[dotted] (0,0) -- (1,0);
  \draw[dotted] (0,0) --(0.5,-0.25) --(1,0) --(1,0.5) --(0,0);
  \draw[dotted] (0.5,-0.25) --(1,0.5);
  
  \node[black, below left] at (0.16666666, -0.08333333) {$\Mat{c}{1 \\ 0 \\ 0 \\ n+1}$};
  \node[black, below] at (0.33333333, -0.16666666) {$\Mat{c}{1 \\ 0 \\ 0 \\ n}$};
  \node[black, above left] at (0.33333333, 0.16666666) {$\Mat{c}{1 \\ 1 \\ 1 \\ n+1}$};
  \node[black, right] at (0.4, 0) {$\Mat{c}{1 \\ 1 \\ 0 \\ n+1}$};
    
  \draw[thick] (0.33333333, 0.16666666) --(0.33333333, -0.16666666) --(0.16666666, -0.08333333) --(0.33333333, 0.16666666);
  \draw[thick] (0.33333333, 0.16666666) --(0.4,0) --(0.33333333, -0.16666666);
  \draw[thick, dashed] (0.4,0) --(0.16666666, -0.08333333);
    
  \node[black] at (0.65,-0.35) {$\mathbb{A}_n$};
  \node[white] at (1.5,0) {.};

 \end{tikzpicture}
  \begin{tikzpicture}[scale=5.0]
  \draw[fill=blue!10!white, draw=none] (0.33333333, 0.16666666) --(0.4, 0) --(0.33333333, -0.16666666) --(0.33333333, 0.16666666);
  \draw[dotted] (0,0) -- (0.4,0);
  \draw[dotted] (0.65,0) -- (1,0);
  \draw[dotted] (0,0) --(0.5,-0.25) --(1,0) --(1,0.5) --(0,0);
  \draw[dotted] (0.5,-0.25) --(1,0.5);
  
  \node[black, below] at (0.33333333, -0.16666666) {$\Mat{c}{1 \\ 0 \\ 0 \\ n}$};
  \node[black, above] at (0.33333333, 0.16666666) {$\Mat{c}{1 \\ 1 \\ 1 \\ n+1}$};
  \node[black, left] at (0.2, 0) {$\Mat{c}{1 \\ 1 \\ 0 \\ n+1}$};
  \node[black, right] at (0.66666666, 0) {$\Mat{c}{1 \\ 1 \\ 0 \\ n}$};
    
  \draw[thick] (0.33333333, 0.16666666) --(0.33333333, -0.16666666) --(0.65, 0) --(0.33333333, 0.16666666);
  \draw[thick, dashed] (0.4,0) --(0.33333333, 0.16666666);
  \draw[thick, dashed] (0.4,0) --(0.33333333, -0.16666666);
  \draw[thick, dashed] (0.4,0) --(0.65, 0);
    
  \node[black] at (0.65,-0.35) {$\mathbb{B}_n$};

 \end{tikzpicture}
  \begin{tikzpicture}[scale=5.0]
  \draw[fill=blue!10!white, draw=none] (0.33333333, 0.16666666) --(0.65, 0) --(0.33333333, -0.16666666) --(0.33333333, 0.16666666);
  \draw[dotted] (0,0) -- (1,0);
  \draw[dotted] (0,0) --(0.5,-0.25) --(1,0) --(1,0.5) --(0,0);
  \draw[dotted] (0.5,-0.25) --(1,0.5);
  
  \node[black, below] at (0.33333333, -0.16666666) {$\Mat{c}{1 \\ 0 \\ 0 \\ n}$};
  \node[black, above left] at (0.33333333, 0.16666666) {$\Mat{c}{1 \\ 1 \\ 1 \\ n+1}$};
  \node[black, above] at (0.66666666, 0.33333333) {$\Mat{c}{1 \\ 1 \\ 1 \\ n}$};
  \node[black, right] at (0.66666666, 0) {$\Mat{c}{1 \\ 1 \\ 0 \\ n}$};
    
  \draw[thick] (0.33333333, -0.16666666) --(0.65, 0) --(0.66666666, 0.33333333) --(0.33333333, 0.16666666) --(0.33333333, -0.16666666) --(0.66666666, 0.33333333);
  \draw[thick, dashed] (0.65, 0) --(0.33333333, 0.16666666);
    
  \node[black] at (0.65,-0.35) {$\mathbb{C}_n$};

 \end{tikzpicture}
 
 \caption{Partition of the base simplex for the 3-dimensional Gauss Map}
 \label{gauss_3d}
\end{center}
\end{figure}

In this case we also get a suspension of $SL(2, \mathbb{Z})$ in $\mathcal{M}_3 < SL(4, \mathbb{Z})$, generated by the matrices $\mathbf{A}_n$ and given by 
\[\mat{cc}{a & b \\ c & d} \mapsto \mat{cccc}{a & 0 & 0 & b \\ 0 & 1 & 0 & 0 \\ 0 & 0 & 1 & 0 \\ c & 0 & 0 & d}\]

\begin{question}
 Is there a subgroup of $\mathcal{M}_3$ homomorphic to $\mathcal{M}_2$?
\end{question}

\change

In order to have a preferred itinerary for the points in the faces and edges of the sub-simplexes we will take the convention that the tetrahedra $\mathbb{A}_n$, $\mathbb{B}_n$ and $\mathbb{C}_n$ do not include the shaded faces shown in Figure \ref{gauss_3d}. 

The base tetrahedron has six edges, namely
\[ \mathsf{A} = \Mat{cc}{0 & 1 \\ 0 & 0 \\ 0 & 0 \\ 1 & 1},
   \mathsf{B} = \Mat{cc}{0 & 1 \\ 0 & 1 \\ 0 & 0 \\ 1 & 1},
   \mathsf{C} = \Mat{cc}{0 & 1 \\ 0 & 1 \\ 0 & 1 \\ 1 & 1}, \] \[
   \mathsf{D} = \Mat{cc}{1 & 1 \\ 0 & 1 \\ 0 & 0 \\ 1 & 1},
   \mathsf{E} = \Mat{cc}{1 & 1 \\ 0 & 1 \\ 0 & 1 \\ 1 & 1} \text{ and }
   \mathsf{F} = \Mat{cc}{1 & 1 \\ 1 & 1 \\ 0 & 1 \\ 1 & 1}.
\]

The edge $\mathsf{A}$ is contained in the pieces $\mathbb{A}_n$ and it is mapped onto itself as a 1-dimensional Gauss Map, i.e. \[ \Mat{c}{x \\ 0 \\ 0 \\ 1} \in \mathsf{A}, \hspace*{2em} G\Mat{c}{x \\ 0 \\ 0 \\ 1} = \Mat{c}{1 - n x \\ 0 \\ 0 \\ x} \text{ where } \frac{1}{n+1} < x \leq \frac{1}{n},\] 
and projected it takes the form $G(x,0,0) = \left( 1/x \mod 1, 0, 0 \right)$.

The edge $\mathsf{B}$ is contained in the pieces $\mathbb{B}_n$ and it is mapped onto the edge $\mathsf{A}$ as a 1-dimensional Gauss Map, i.e. \[ \Mat{c}{x \\ x \\ 0 \\ 1} \in \mathsf{B}, \hspace*{2em} G\Mat{c}{x \\ x \\ 0 \\ 1} = \Mat{c}{1 - n x \\ 0 \\ 0 \\ x} \text{ where } \frac{1}{n+1} < x \leq \frac{1}{n},,\] 
and projected it takes the form $G(x,x,0) = \left( 1/x \mod 1, 0, 0 \right)$.

The edge $\mathsf{C}$ is contained in the pieces $\mathbb{C}_n$ and it is mapped onto the edge $\mathsf{A}$ as a 1-dimensional Gauss Map, i.e. \[ \Mat{c}{x \\ x \\ x \\ 1} \in \mathsf{C}, \hspace*{2em} G\Mat{c}{x \\ x \\ x \\ 1} = \Mat{c}{1 - n x \\ 0 \\ 0 \\ x} \text{ where } \frac{1}{n+1} < x \leq \frac{1}{n},,\] 
and projected it takes the form $G(x,x,x) = \left( 1/x \mod 1, 0, 0 \right)$.

The edge $\mathsf{D}$ is contained in the piece $\mathbb{B}_1$ and it is mapped onto the edge $\mathsf{C}$ in a reversed order, i.e. \[ \Mat{c}{1 \\ x \\ 0 \\ 1} \in \mathsf{D}, \hspace*{2em} G\Mat{c}{1 \\ x \\ 0 \\ 1} = \Mat{c}{1-x \\ 1-x \\ 1-x \\ 1},\] 
and projected it takes the form $G(1,x,0) = \left( 1-x, 1-x, 1-x \right)$.

The edge $\mathsf{E}$ is contained in the piece $\mathbb{C}_1$ and it is mapped onto the edge $\mathsf{B}$ in a reversed order, i.e. \[ \Mat{c}{1 \\ x \\ x \\ 1} \in \mathsf{E}, \hspace*{2em} G\Mat{c}{1 \\ x \\ x \\ 1} = \Mat{c}{1-x \\ 1-x \\ 0 \\ 1},\] 
and projected it takes the form $G(1,x,x) = \left( 1-x, 1-x, 0 \right)$.

The edge $\mathsf{F}$ is contained in the piece $\mathbb{C}_1$ and it is mapped onto the edge $\mathsf{C}$ in a reversed order, i.e. \[ \Mat{c}{1 \\ 1 \\ x \\ 1} \in \mathsf{F}, \hspace*{2em} G\Mat{c}{1 \\ 1 \\ x \\ 1} = \Mat{c}{1-x \\ 1-x \\ 1-x \\ 1},\] 
and projected it takes the form $G(1,1,x) = \left( 1-x, 1-x, 1-x \right)$.

From this we can see that any point on a preimage of an edge will be eventually mapped to the edge $\mathsf{A}$ where the dynamics are the same as the 1-dimensional Gauss Map. 
Now consider the four faces of the base tetrahedron, namely
\[
 \mathsf{AB} = \Mat{ccc}{0 & 1 & 1 \\ 0 & 0 & 1 \\ 0 & 0 & 0 \\ 1 & 1 & 1}, 
 \mathsf{BC} = \Mat{ccc}{0 & 1 & 1 \\ 0 & 1 & 1 \\ 0 & 0 & 1 \\ 1 & 1 & 1}, 
 \mathsf{AC} = \Mat{ccc}{0 & 1 & 1 \\ 0 & 0 & 1 \\ 0 & 0 & 1 \\ 1 & 1 & 1} \text{ and } 
 \mathsf{FR} = \Mat{ccc}{1 & 1 & 1 \\ 0 & 1 & 1 \\ 0 & 0 & 1 \\ 1 & 1 & 1}.
\]

To parametrize the interior of the faces we will take pairs of real numbers $(x, y)$ such that $0 < y < x < 1$. With these parameters a point in the face $\mathsf{AB}$ is of the form $(x,y)_{\mathsf{AB}} = (x, y, 0)$, a point in the face $\mathsf{AC}$ is of the form $(x,y)_{\mathsf{AC}} = (x, y, y)$, a point in the face $\mathsf{BC}$ is of the form $(x,y)_{\mathsf{BC}} = (x, x, y)$, and a point in the face $\mathsf{FR}$ is of the form $(x,y)_{\mathsf{FR}} = (1, x, y)$.

The frontal face $\mathsf{FR}$ is contained in the piece $\mathbb{C}_1$ and it is mapped to the face $\mathsf{BC}$ as follows 
\[\Mat{c}{1 \\ x \\ y \\ 1} \in \mathsf{FR}, \hspace*{2em} G \Mat{c}{1 \\ x \\ y \\ 1} = \Mat{c}{1-y \\ 1-y \\ x-y \\ 1}, \hspace*{1em} \text{ we write it as } \hspace*{1em} (x,y)_{\mathsf{FR}} \mapsto (1-y, x-y)_{\mathsf{BC}} .\]

The faces $\mathsf{AB}$, $\mathsf{BC}$ and $\mathsf{AC}$ have its own sub partitions in a similar fashion as the triangle for the 2-dimensional case (see Figures \ref{gauss_2d} and \ref{parameter_faces}), and the dynamics are related. 
For the partition in the 2-dimensional case, let's call \emph{lower type} triangles those which share an edge with the base of the triangle, and \emph{upper type} triangles those which share an edge with the diagonal side of the triangle.
All triangles of the same type are mapped in a similar way to the entire triangle; the lower type triangles are \emph{reflected} vertically and scaled, and the upper type triangles are \emph{rotated} and scaled. 
More precisely, the map on a lower type triangle takes the form \[(x,y) \mapsto \left( \frac{1}{x} \mod 1, \frac{y}{x}\right),\] and the map on a upper type triangle takes the form \[(x,y) \mapsto \left( \frac{1-y}{x} \mod 1, 1-\frac{y}{x}\right).\]
The dynamics on the 3-dimensional case are similar.
A face will have a partition in upper and lower type triangles (see Figure \ref{parameter_faces}), and these pieces will be \emph{reflected} or \emph{rotated} respectively and scaled to fit a face which might not be the one they come from.  

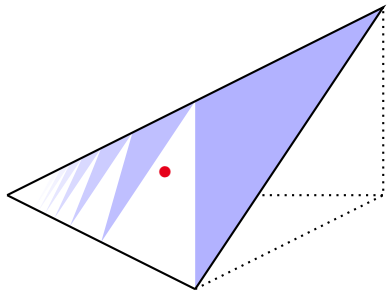
\begin{figure}
 \begin{center}
   \begin{tikzpicture}[scale=5.0]
  \draw[fill=blue!30!white, draw=none] (0.5,0.5) --(1,0) --(1,1) --(0.5,0.5);
  \draw[fill=blue!25!white, draw=none] (0.5,0.5) --(0.5,0) --(0.33333333,0.33333333) --(0.5,0.5);
  \draw[fill=blue!20!white, draw=none] (0.25,0.25) --(0.33333333,0) --(0.33333333,0.33333333) --(0.25,0.25);
  \draw[fill=blue!15!white, draw=none] (0.25,0.25) --(0.25,0) --(0.2,0.2) --(0.25,0.25);
  \draw[fill=blue!10!white, draw=none] (0.16666666,0.16666666) --(0.2,0) --(0.2,0.2) --(0.16666666,0.16666666);
  \draw[fill=blue!5!white, draw=none] (0.16666666,0.16666666) --(0.16666666,0) --(0.14285714,0.14285714) --(0.16666666,0.16666666);
  \draw[fill=blue!2!white, draw=none] (0.125,0.125) --(0.14285714,0) --(0.14285714,0.14285714) --(0.125,0.125);
   
  \draw[thick] (0,0) --(1,0) --(1,1) --(0,0);
  
  \node[red!90!blue] at (0.56, 0.2690121944) {\textbullet};
  \node[black] at (0.5, -0.15) {$(x, y)$};
 \end{tikzpicture}
 
 \begin{tikzpicture}[scale=5.0]
  \draw[fill=blue!30!white, draw=none] (0.5,0) --(1,0) --(0.5,-0.25) --(0.5,0);
  \draw[fill=blue!25!white, draw=none] (0.25, -0.125) --(0.5,0) --(0.33333333,0) --(0.25, -0.125);
  \draw[fill=blue!20!white, draw=none] (0.25,0) --(0.33333333,0) --(0.16666666,-0.08333333) --(0.25,0);
  \draw[fill=blue!15!white, draw=none] (0.125,-0.0625) --(0.25,0) --(0.2,0) --(0.125,-0.0625);
  \draw[fill=blue!10!white, draw=none] (0.16666666,0) --(0.2,0) --(0.1,-0.05) --(0.16666666,0);
  \draw[fill=blue!5!white, draw=none] (0.08333333,-0.04166666) --(0.16666666,0) --(0.14285714,0) --(0.08333333,-0.04166666);
  \draw[fill=blue!2!white, draw=none] (0.125,0) --(0.14285714,0) --(0.0714285714,-0.0357142857) --(0.125,0);
 
  \draw[thick] (0,0) -- (1,0);
  \draw[thick] (0,0) --(0.5,-0.25) --(1,0);
  \draw[thick, dotted] (0.5,-0.25) --(1,0.5) --(0,0);
  \draw[dotted, thick] (1,0) --(1,0.5);
  
  \node[red!90!blue] at (0.42, -0.08) {\textbullet};
  \node[black] at (0.5, -0.35) {$(x,y)_{\mathsf{AB}} = (x, y, 0) \in \mathsf{AB}$};
  \node[white] at (1.5, 1) {};
 \end{tikzpicture}
 \begin{tikzpicture}[scale=5.0]
  \draw[thick, dotted] (0.6,0) -- (1,0);
  \draw[fill=blue!30!white, draw=none] (0.5,0.25) --(1,0.5) --(0.5,-0.25) --(0.5,0.25);
  \draw[fill=blue!25!white, draw=none] (0.25, -0.125) --(0.5,0.25) --(0.33333333,0.16666666) --(0.25, -0.125);
  \draw[fill=blue!20!white, draw=none] (0.25,0.125) --(0.33333333,0.16666666) --(0.16666666,-0.08333333) --(0.25,0.125);
  \draw[fill=blue!15!white, draw=none] (0.125,-0.0625) --(0.25,0.125) --(0.2,0.1) --(0.125,-0.0625);
  \draw[fill=blue!10!white, draw=none] (0.16666666,0.08333333) --(0.2,0.1) --(0.1,-0.05) --(0.16666666,0.08333333);
  \draw[fill=blue!5!white, draw=none] (0.08333333,-0.04166666) --(0.16666666,0.08333333) --(0.14285714,0.0714285714) --(0.08333333,-0.04166666);
  \draw[fill=blue!2!white, draw=none] (0.125,0.0625) --(0.14285714,0.0714285714) --(0.0714285714,-0.0357142857) --(0.125,0.0625);
 
  \draw[thick] (0,0) --(0.5,-0.25);
  \draw[thick] (0.5,-0.25) --(1,0.5) --(0,0);
  \draw[dotted, thick] (0.5,-0.25) --(1,0) --(1,0.5);
  \node[red!90!blue] at (0.42, 0.05944) {\textbullet};
  \node[black] at (0.5, -0.35) {$(x,y)_{\mathsf{AC}} = (x, y, y) \in \mathsf{AC}$};
 \end{tikzpicture} 
 \end{center}
 \caption{Parametrization of the faces $\mathsf{AB}$ and $\mathsf{AC}$ with \emph{upper} and \emph{lower} type pieces indicated}
 \label{parameter_faces}
\end{figure}

The partition on the face $\mathsf{AB}$ has symbols $\mathbb{A}_n$ on the lower side and symbols $\mathbb{B}_n$ on the upper side (compare Figures \ref{gauss_3d} and \ref{parameter_faces}). 
The triangles with symbol $\mathbb{A}_n$ are \emph{reflected} and mapped to the face $\mathsf{AB}$.
The triangles with symbol $\mathbb{B}_n$ are \emph{rotated} and mapped to the face $\mathsf{AC}$.
On \emph{face coordinates} this looks like \[ (x,y)_{\mathsf{AB}} \xrightarrow{\mathbb{A}_n} \left( \frac{1}{x} \mod 1, \frac{y}{x}\right)_{\mathsf{AB}}, \hspace*{1em} \text{and} \hspace*{1em} (x,y)_{\mathsf{AB}} \xrightarrow{\mathbb{B}_n} \left( \frac{1-y}{x} \mod 1, 1-\frac{y}{x}\right)_{\mathsf{AC}}.\]

The partition on the face $\mathsf{BC}$ has symbols $\mathbb{B}_n$ on the lower side and symbols $\mathbb{C}_n$ on the upper side. 
The triangles with symbol $\mathbb{B}_n$ are \emph{reflected} and mapped to the face $\mathsf{AB}$.
The triangles with symbol $\mathbb{C}_n$ are \emph{rotated} and mapped to the face $\mathsf{AC}$.

The partition on the face $\mathsf{AC}$ has symbols $\mathbb{A}_n$ on the lower side and symbols $\mathbb{C}_n$ on the upper side. 
The triangles with symbol $\mathbb{A}_n$ are \emph{reflected} and mapped to the face $\mathsf{AC}$.
The triangles with symbol $\mathbb{C}_n$ are \emph{rotated} and mapped to the face $\mathsf{AB}$.

The dynamics of the faces can follow any infinite path in the graph shown in Figure \ref{faces_dynamics}. 
This graph shows that any point in a preimage of a face will eventually be \emph{jumping} between the faces $\mathsf{AB}$ and $\mathsf{AC}$ following similar dynamics as on the 2-dimensional case, but on two different triangles.
Note that this includes itineraries consisting only of symbols $\mathbb{A}_n$, which again corresponds to the dynamics on the edges. 
This means that, in this \emph{jumping} dynamics, points that stay in one face are actually in the common edge $\mathsf{A}$.
Note also that, with the convention on the pieces, there are no defined preimages of the frontal face $\mathsf{FR}$.


\begin{figure}
\begin{center}
 \begin{tikzpicture}[scale=5.0]
  
  \draw[black] (0, 0) circle (0.1);
  \node[black] at (0, 0) {$\mathsf{AB}$};
  \draw[black] (1, 0) circle (0.1);
  \node[black] at (1, 0) {$\mathsf{AC}$};
  \draw[black] (0.5, -0.5) circle (0.1);
  \node[black] at (0.5,-0.5) {$\mathsf{BC}$};
  \draw[black] (0, -0.5) circle (0.1);
  \node[black] at (0, -0.5) {$\mathsf{FR}$};
  
  \draw[<-, thick] (0.15, 0) --(0.85,0);
  \draw[->, thick] (0.1,0.1) to [out=45,in=135] 
                   (0.9,0.1);
  \draw[->, thick] (-0.1,-0.1) to [out=225,in=270] 
                   (-0.4,0) to [out=90,in=135] 
                   (-0.1,0.1);
  \draw[->, thick] (1.1,0.1) to [out=45,in=90] 
                   (1.4,0) to [out=270,in=315] 
                   (1.1,-0.1);
  \draw[->, thick] (0.4, -0.4) --(0.1,-0.1);
  \draw[->, thick] (0.6, -0.4) --(0.9,-0.1);
  \draw[->, thick] (0.15, -0.5) --(0.35,-0.5);
  
  \node[black] at (-0.5,0) {$\mathbb{A}_n$};
  \node[black] at (1.5,0) {$\mathbb{A}_n$};
  \node[black, above] at (0.5,0.3) {$\mathbb{B}_n$};
  \node[black, above] at (0.5,0) {$\mathbb{C}_n$};
  \node[black] at (0.35,-0.25) {$\mathbb{B}_n$};
  \node[black] at (0.65,-0.25) {$\mathbb{C}_n$};
  \node[black] at (0.25,-0.45) {$\mathbb{C}_1$};
  
 \end{tikzpicture}
\end{center}
 \caption{Sub-shift of the face dynamics for the 3-dimensional Gauss Map.}
 \label{faces_dynamics}
\end{figure}
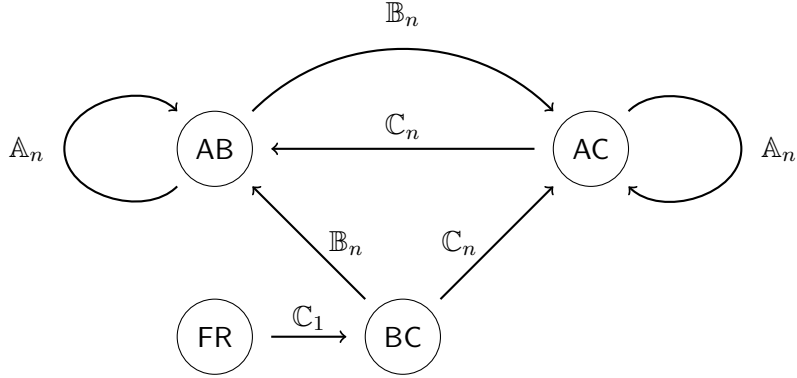

As in previous cases, we have that rational points approximate non-rational points.

\begin{prop}
 Points with rational coordinates are preimages of zero.
\end{prop}

\begin{proof}[Sketch of the proof]
 The proof for the 2-dimensional case can be generalized to any dimension:
 A point with rational coordinates is represented projectively by a vector $\mathbf{r} \in V$ with integer entries that satisfies the inequalities that define the base simplex $V$.
 Suppose $\mathbf{r}$ is an interior point of $V$, then the inequalities are strict. 
 Note that since $G$ preserves $V$, it also preserves the inequalities. 
 Also, since the defining matrices 
 have integer entries, they also keep the entries of $\mathbf{r}$ integer.
 
 Now suppose that $G^{k} (\mathbf{r})$ is a interior point for all $k \geq 0$. 
 Again, we can find that the entries of $G^{k} (\mathbf{r})$ form a strictly decreasing sequence of integers, and therefore they converge to zero, which would be a contradiction to $G^{k} (\mathbf{r})$ being a interior point for all $k \geq 0$.
 Some of the entries would be the first to get to zero, meaning that $\mathbf{r}$ is mapped to a lower dimensional facet of the simplex, where the sequences are no longer strictly decreasing. 
 
 By adapting the result on the previous dimension, the rational points are mapped to zero. 
\end{proof} 

Let $\mathbf{v} \in V$ be a non-rational point with itinerary $\mathbb{X}_{a_1}, \mathbb{X}_{a_2}, ...$, where $\mathbb{X}$ can be $\mathbb{A}$, $\mathbb{B}$ or $\mathbb{C}$ and $a_k > 0$.
Again, for each $n \geq 1$ let $\mathbf{I}_{\mathbf{v},n} = \left( \mathbf{X}_{a_1}^{-1} \mathbf{X}_{a_2}^{-1} \cdots \mathbf{X}_{a_n}^{-1} \right)$, where $\mathbf{X}$ is $\mathbf{A}$, $\mathbf{B}$ or $\mathbf{C}$ as indicated by the itinerary. 
The point $\mathbf{v}$ is in the simplex $\mathbb{I}_{n}(\mathbf{v}) = \mathbf{I}_{\mathbf{v},n} V$.
We have again that $\mathbb{I}_{n+1}(\mathbf{v}) \subset \mathbb{I}_{n}(\mathbf{v})$ and  $\mathbf{v} = \bigcap_{n \geq 1} \mathbb{I}_{n}(\mathbf{v})$.
Again, we call the sequence $\mathbb{I}_{n}(\mathbf{v})$ the approximating simplexes of $\mathbf{v}$.

\begin{prop}
 A non-rational point $\mathbf{v} \in V$ has an itinerary consisting only of symbols of type $\mathbb{A}_n$ if and only if it is a point in the edge $\mathsf{A}$.
 Moreover, if the itinerary is $\mathbb{A}_{a_1}, \mathbb{A}_{a_2}, ...$ then the irrational number on the first coordinate is given in continued fractions as $[0: a_1, a_2, ...]$.
\end{prop}

\begin{proof}
 The second part of the statement follows directly from the previous discussion on the dynamics restricted to the edge $\mathsf{A}$.
 For the first part, note that if a point is in the edge $\mathsf{A}$ then it is in an edge of some sub-simplex $\mathbb{A}_n$ and it is mapped again to the edge $\mathsf{A}$, and therefore its itinerary consists only of $\mathbb{A}_n$ symbols.
 
 Now suppose $\mathbf{v}$ has the itinerary $\mathbb{A}_{a_1}, \mathbb{A}_{a_2}, ...$, then by previous results the approximating simplexes of $\mathbf{v}$ are of the form 
 \[ \mathbb{I}_{n}(\mathbf{v}) = 
    \mat{cccc}{p_{n-1} & 0 & 0 & p_n \\ 0 & 1  & 0 & 0 \\ 0 & 0 & 1 & 0\\ q_{n-1} & 0 & 0 & q_n} V = 
    \Mat{cccc}{p_{n-1} & p_{n-1} + p_{n} & p_{n-1} + p_{n} & p_{n-1} + p_{n} \\ 0 & 0 & 1 & 1 \\ 0 & 0 & 0 & 1 \\ q_n & q_{n-1} + q_n & q_{n-1} + q_n & q_{n-1} + q_n},\]
 where $p_n/q_n$ are the convergents of the irrational number $[0: a_1, a_2, ...]$.
 This simplexes projected to $\mathbb{R}^3$ are tetrahedra with vertices 
 \[ \left( \frac{p_n}{q_n}, 0 , 0\right), \hspace*{1em} 
    \left( \frac{p_{n-1} + p_{n}}{q_{n-1} + q_n}, 0 , 0\right), \hspace*{1em} 
    \left( \frac{p_{n-1} + p_{n}}{q_{n-1} + q_n}, \frac{1}{q_{n-1} + q_n}  , 0\right), \] \[ \text{and} \hspace*{1em} 
    \left( \frac{p_{n-1} + p_{n}}{q_{n-1} + q_n}, \frac{1}{q_{n-1} + q_n}, \frac{1}{q_{n-1} + q_n}  \right).\]
 When projected to $\mathbb{R}^3$, the point $\mathbf{v} = (x, y, z)$ is in all of these tetrahedra. 
 Coordinate-wise, $x$ is approximated by the convergents $p_n/q_n$ and $y, z \leq (q_{n-1} + q_n)^{-1} \rightarrow 0$.
 Therefore, $(x, y, z) = ([0: a_1, a_2, ...], 0, 0)$.
\end{proof}
 
\begin{coro}
 Points whose itinerary has a tail of only symbols of type $\mathbb{A}_n$ are preimages of points in $\mathsf{A}$.
\end{coro}

\begin{prop}
 A point has an itinerary with the pattern described by the graph on Figure \ref{characterization_faces_dynamics} if and only if it is on the face $\mathsf{AB}$ or $\mathsf{AC}$, depending on the label of the starting node.
\end{prop}

Note that the graph on Figure \ref{characterization_faces_dynamics} includes the case of points in the edge $\mathsf{A}$.

\begin{proof}
 The pattern described by the graph consists of streaks of symbols $\mathbb{A}_n$ followed by one symbol of type $\mathbb{B}_n$ or $\mathbb{C}_n$, and when one of these appears, the next \textemdash{}if there is any\textemdash{} has to be of the other type.
 By our previous discussion, the itinerary of a point in one of the faces $\mathsf{AB}$ or $\mathsf{AC}$ has this pattern.
 
 \begin{figure}
 \begin{center}
 \begin{tikzpicture}[scale=5.0]
  
  \draw[black] (0, 0) circle (0.1);
  \node[black] at (0, 0) {$\mathsf{AB}$};
  \draw[black] (1, 0) circle (0.1);
  \node[black] at (1, 0) {$\mathsf{AC}$};
  
  \draw[->, thick] (0.1,0.1) to [out=45,in=135] 
                   (0.9,0.1);
  \draw[<-, thick] (0.1,-0.1) to [out=315,in=225] 
                   (0.9,-0.1);
  \draw[->, thick] (-0.1,-0.1) to [out=225,in=270] 
                   (-0.4,0) to [out=90,in=135] 
                   (-0.1,0.1);
  \draw[->, thick] (1.1,0.1) to [out=45,in=90] 
                   (1.4,0) to [out=270,in=315] 
                   (1.1,-0.1);
  
  \node[black] at (-0.5,0) {$\mathbb{A}_n$};
  \node[black] at (1.5,0) {$\mathbb{A}_n$};
  \node[black, above] at (0.5,0.3) {$\mathbb{B}_n$};
  \node[black, below] at (0.5,-0.3) {$\mathbb{C}_n$};
  
 \end{tikzpicture}
 \end{center}
 \caption{Characterizing sub-shift for the 2-dimensional dynamics in the 3-dimensional Gauss Map.}
 \label{characterization_faces_dynamics}
 \end{figure}
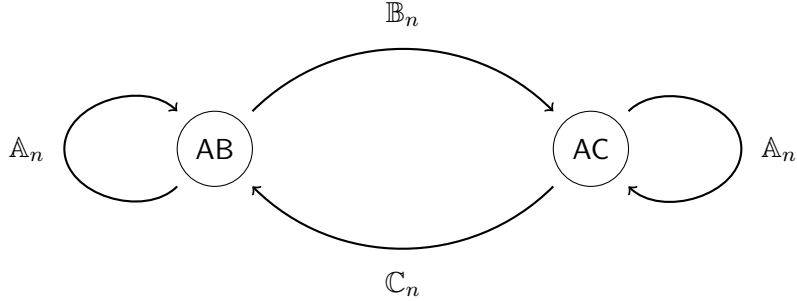

 Let $\mathbf{v}$ be a point with itinerary described by the graph on Figure \ref{characterization_faces_dynamics}.
 Suppose that the starting node is $\mathsf{AB}$, the other case is similar.
 
 To each finite streak $\mathbb{A}_{a_1} \mathbb{A}_{a_2} \cdots \mathbb{A}_{a_k}$ we associate the rational number $\frac{p}{q} = [0; a_1, a_2, ..., a_k]$. 
 If the streak is infinite, $\mathbb{A}_{a_1} \mathbb{A}_{a_2} \cdots$, we associate the irrational number $\alpha = [0; a_1, a_2, ...]$, and we write the infinite streak as a symbol $\mathbb{A}_{\alpha}$.
 For finite streaks we have the matrix \[\left( \mathbf{A}_{a_1}^{-1}  \mathbf{A}_{a_2}^{-1} \cdots \mathbf{A}_{a_k}^{-1} \right) = \mat{cccc}{r&0&0&p \\ 0&1&0&0 \\ 0&0&1&0 \\ s&0&0&q},\] where $\frac{r}{s} = [0; a_1, a_2, ..., a_{k-1}]$.
 For empty streaks the associated rational number will be $0$ and the matrix will be the identity.
 
 Each finite streak is followed by a symbol $\mathbb{B}_n$ or $\mathbb{C}_n$, we will denote \[ \breve{\mathbb{B}}_{\frac{p}{q}, n}  = \mathbb{A}_{a_1} \mathbb{A}_{a_2} \cdots \mathbb{A}_{a_k} \mathbb{B}_{n} \hspace*{2em} \text{and} \hspace*{2em}  \breve{\mathbb{C}}_{\frac{p}{q}, n}  = \mathbb{A}_{a_1} \mathbb{A}_{a_2} \cdots \mathbb{A}_{a_k} \mathbb{C}_{n}.\]
 For these words the corresponding matrices are \[ \breve{\mathbf{B}}_{\frac{p}{q}, n}^{-1} := \mat{cccc}{p&0&-p&r+np \\ 0&0&-1&1 \\ 0&1&-1&0 \\ q&0&-q&s+nq} \hspace*{1em} \text{and} \hspace*{1em} \breve{\mathbf{C}}_{\frac{p}{q}, n}^{-1} := \mat{cccc}{p&-p&0&r+np \\ 0&-1&1&1 \\ 0&-1&0&1 \\ q&-q&0&s+nq}.\]
 
 We can rewrite the itinerary of $\mathbf{v}$ either as infinitely many of the symbols $\breve{\mathbb{B}}_{\frac{p}{q}, n}$ and $ \breve{\mathbb{C}}_{\frac{p}{q}, n}, $ or as a finite amount of those symbols followed by a $\mathbb{A}_{\alpha}$.
 Suppose we have the second case, the itinerary is $\breve{\mathbb{B}}_{\frac{p_1}{q_1}, n_1} \breve{\mathbb{C}}_{\frac{p_2}{q_2}, n_2} \cdots \breve{\mathbb{X}}_{\frac{p_j}{q_J}, n_j} \mathbb{A}_{\alpha}$, and the point $\mathbf{v}$ is a preimage of the point \[ \mathbf{v}_j = \Mat{c}{\alpha\\0\\0\\1} \in \mathsf{A}.\]
 We trace back the point to the face $\mathsf{AB}$.
 For the $j$-th symbol of the rewritten itinerary we have \[ \breve{\mathbf{B}}_{\frac{p}{q}, n}^{-1} \mathbf{v}_j = \Mat{c}{p(\alpha+n) + r \\1\\0\\ q(\alpha+n) + s}  \in \mathsf{AB}, \hspace*{1em} \text{and} \hspace*{1em} \breve{\mathbf{C}}_{\frac{p}{q}, n}^{-1} \mathbf{v}_j = \Mat{c}{p(\alpha+n) + r \\1\\1\\ q(\alpha+n) + s}  \in \mathsf{AC}. \]
 Now, for any point $\mathbf{r} \in \mathsf{AC}$ the previous symbol must be of the type $\breve{\mathbb{B}}_{\frac{p}{q}, n}$ and we have that \[ \breve{\mathbf{B}}_{\frac{p}{q}, n}^{-1} \mathbf{r} = \breve{\mathbf{B}}_{\frac{p}{q}, n}^{-1} \Mat{c}{x\\y\\y\\1} = \Mat{c}{p(x-y+n) + r \\1-y\\0\\ q(x-y+n) + s}  \in \mathsf{AB}. \]
 Similarly, for any point $\mathbf{r} \in \mathsf{AB}$ the previous symbol must be of the type $\breve{\mathbb{C}}_{\frac{p}{q}, n}$ and we get $\breve{\mathbf{C}}_{\frac{p}{q}, n}^{-1} \mathbf{r} \in \mathsf{AC}$.
 So, we have that the preimages of $\mathbf{v}_j$, under the rewritten itinerary, jump between the faces $\mathsf{AB}$ and $\mathsf{AC}$ to end up with $\mathbf{v} \in \mathsf{AB}$ since the first symbol is $\breve{\mathbb{B}}_{\frac{p_1}{q_1}, n_1}$. 
 
 Now suppose that the itinerary of $\mathbf{v}$ consists of infinitely many streaks.
 As we will see, all the approximating simplexes $\mathbb{I}_{n}(\mathbf{v})$ under the rewritten itinerary have a face contained in $\mathsf{AB}$ and that implies that $\mathbf{v} \in \mathsf{AB}$. If that were not the case, by convexity the set \[\displaystyle\bigcap_{n \geq 0}\mathbb{I}_{n}(\mathbf{v})\] would contain a segment connecting $\mathbf{v}$ and some point in $\mathsf{AB}$, and that would contradict the uniqueness of itineraries.

 For any rational $p/q$ and any $n>0$ the simplex $\breve{\mathbf{B}}_{\frac{p}{q},  n }^{-1} V$ has a face contained in $\mathsf{AB}$ and the simplex $\breve{\mathbf{C}}_{\frac{p}{q},  n }^{-1} V$ has a face contained in $\mathsf{AC}$.
 As we have seen, $ \breve{\mathbf{B}}_{\frac{p}{q}, n}^{-1}$ maps points in $\mathsf{AC}$ to $\mathsf{AB}$ and $\breve{\mathbf{C}}_{\frac{p}{q}, n}^{-1}$ maps points in $\mathsf{AB}$ to $\mathsf{AC}$, so all approximating simplexes have a face contained in $\mathsf{AB}$ given that the first symbol is $\breve{\mathbb{B}}_{\frac{p_1}{q_1}, n_1}$. 
\end{proof}

\begin{coro}
 A point with an itinerary that eventually falls in the pattern described by Figure \ref{characterization_faces_dynamics} is a preimage of a point in a face of $V$.
\end{coro}

\begin{question}
 Is there an arithmetic/algebraic relation between the coordinates of a preimage of a face point?
\end{question}

Let $0 < y < x < 1$. For which numbers $z < y$ the point \[\mathbf{v} = \Mat{c}{x\\y\\z\\1} \] is a preimage of a face point?
As we have seen, the preimages of faces are triangles with rational vertices. 
These triangles are subsets of planes defined by linear equations with rational coefficients. 
We can express a number $z$ in one of these preimages as $z = r x + s y + t$ with $r,s,t$ rational.

\begin{conjecture}
 Let \[\mathbf{v} = \Mat{c}{x\\y\\z\\1} \in V.\]
 Consider the vector space over $\mathbb{Q}$ spanned by $\lbrace x, y, z, 1 \rbrace$ and let $d$ be its dimension.
 If $d=1$ then $\mathbf{v}$ is a preimage of a vertex.
 If $d=2$ then $\mathbf{v}$ is a preimage of an edge.
 If $d=3$ then $\mathbf{v}$ is a preimage of a face.
 If $d\geq4$ then $\mathbf{v}$ stays an interior point.
 
 Moreover, if $d=4$ and $x, y, z$ are rational independent then the point is eventually periodic. 
\end{conjecture}

\begin{question}
 Let $\alpha, \beta$ be quadratic irrationals. Do the following points satisfy the previous conjecture? 
 \[ \Mat{c}{\alpha \\ \beta \\ \alpha^2 \\ 1}, \hspace*{2em} \Mat{c}{\alpha \\ \beta \\ \beta^2 \\ 1}, \hspace*{2em} \Mat{c}{\alpha \\ \beta \\ \alpha - \beta \\ 1}, \hspace*{2em} \Mat{c}{\alpha \\ \beta \\ \alpha \beta \\ 1}.\]
\end{question}

\section{The $n$-dimensional case}

\begin{conjecture}
 There is a cyclic action of the matrix $A$ in the edges of $V$. 
\end{conjecture}

\begin{conjecture}
 There are $n$ families of domains for the first return map $G$.
\end{conjecture}

\begin{conjecture}
 There is a subgroup of $\mathcal{M}_n$ homomorphic to $SL(2,\mathbb{Z})$, generated by the matrices of the corresponding family $\mathbf{A}_n$.
\end{conjecture}

\begin{question}
 How \emph{big} is $\mathcal{M}_n$ inside $SL(n+1, \mathbb{Z})$? 
\end{question}

\begin{question}
 Are there homomorphisms $\mathcal{M}_{n} \rightarrow \mathcal{M}_{n+1}$ ?
\end{question}

\begin{conjecture}
 Rational points are preimages of zero.
\end{conjecture}

\begin{conjecture}
 The approximating simplexes are the best approximations.
\end{conjecture}

\begin{conjecture}
 There is a graph that describes the dynamics on $k$ dimensional facets and characterizes the itineraries of their preimages.
\end{conjecture}

\begin{conjecture}
 Consider the vector space over $\mathbb{Q}$ spanned by the coordinates of a point in the projected simplex and let $k$ be the dimension of this space.
 If $k \leq n$ then the point is a preimage of a $(k-1)$-dimensional facet.
 In this case, if the coordinates are in the same number field of degree $k$, then the point is periodic in the $(k-1)$-dimensional dynamics.
 If $k > n$ then the point is interior.
 If $k = n+1$ and the coordinates are in the same number field of degree $n$, then the point is periodic.
\end{conjecture}

\begin{question}
 Can we compute the density function for a $G$-invariant ergodic measure absolutely continuous with respect to Lebesgue measure?
\end{question}

\begin{question}
 Do the Lyapunov Exponents for the n-dimensional $G$ measure the exponential rate of approximation to a non rational point by the approximating simplexes?
\end{question}

\end{document}